\documentclass[letterpaper,11pt,oneside,reqno]{article}

\usepackage{amsmath,amssymb,amsthm,amsfonts, mathtools}
\usepackage[colorlinks=true,linkcolor=blue,citecolor=red]{hyperref}
\usepackage{graphicx,color}
\usepackage{upgreek}
\usepackage[mathscr]{euscript}

\usepackage{todonotes}
\allowdisplaybreaks
\numberwithin{equation}{section}

\usepackage{tikz}
\usetikzlibrary{shapes,arrows,positioning,decorations.markings}

\usepackage{array}
\usepackage{adjustbox}
\usepackage{cleveref}
\usepackage{enumerate}
\usepackage{caption}
\usepackage{esint}
\usepackage[DIV=12]{typearea}
\newcommand{\ssp}{\hspace{1pt}}
\renewcommand{\Re}{\mathop{\mathrm{Re}}}
\renewcommand{\Im}{\mathop{\mathrm{Im}}}
\newcommand\lozv[2]
{\begin{tikzpicture}[baseline=#1,scale=#2]\draw[very thick](0,0)--++(0,1)--++(-1,1)--++(0,-1)--cycle; \end{tikzpicture}}

\newtheorem{proposition}{Proposition}[section]
\newtheorem{lemma}[proposition]{Lemma}

\newtheorem{theorem}[proposition]{Theorem}
\theoremstyle{definition}
\newtheorem{definition}[proposition]{Definition}
\newtheorem{remark}[proposition]{Remark}

\begin{document}
\title{Asymptotics of noncolliding $q$-exchangeable random walks}

\author{Leonid Petrov, Mikhail Tikhonov}

\date{}

\maketitle

\begin{abstract} 
	We consider a process of noncolliding $q$-exchangeable random walks on $\mathbb{Z}$ making steps $0$ (``straight'') and $-1$ (``down''). A single random walk is called $q$-exchangeable if under an elementary transposition of the neighboring steps $ (\textnormal{down},\textnormal{straight}) \to (\textnormal{straight}, \textnormal{down}) $ the probability of the trajectory is multiplied by a parameter $q\in(0,1)$. Our process of $m$ noncolliding $q$-exchangeable random walks is obtained from the independent $q$-exchangeable walks via the Doob's $h$-transform for a certain nonnegative eigenfunction $h$ with the eigenvalue less than $1$.  The system of $m$ walks evolves in the presence of an absorbing wall at $0$.

	We show that the trajectory of the noncolliding $q$-exchangeable walks started from an arbitrary initial configuration forms a determinantal point process, and express its kernel in a double contour integral form. This kernel is obtained as a limit from the correlation kernel of $q$-distributed random lozenge tilings of sawtooth polygons.

	In the limit as $m\to \infty$, $q=e^{-\gamma/m}$ with $\gamma>0$ fixed, and under a suitable scaling of the initial data, we obtain a limit shape of our noncolliding walks and also show that their local statistics are governed by the incomplete beta kernel. The latter is a distinguished translation invariant ergodic extension of the two-dimensional discrete sine kernel.
\end{abstract}


\section{Introduction}
\label{sec:intro}

The main object of the present paper is an ensemble $\Upsilon_m$ of random point configurations in the two-dimensional lattice $\mathbb{Z}^{2}_{\ge0}$ which belongs to two broad classes: \emph{noncolliding random walks} and \emph{$q$-distributed random lozenge tilings}.

\bigskip

The noncolliding random walks on $\mathbb{Z}$ is a Markov
chain of a fixed number $m$ of particles performing
independent simple random walks.  They interact through the
condition that they never collide, which is equivalent to
Coulomb repulsion.  This model can be traced back to
Karlin--McGregor \cite{KMG59-Coincidence}, see
K{}\"onig--O’Connell--Roch \cite{konig2002non} for a
detailed exposition.  The noncolliding random walks are a
discretization of the celebrated $\beta=2$ Dyson Brownian
motion \cite{dyson1962brownian} describing the eigenvalues
of the Gaussian Unitary Ensemble.
More recently, noncolliding random walks for other random
matrix $\beta$ values were considered by Huang
\cite{huang2021beta} and Gorin--Huang
\cite{gorin2022dynamical}.

\smallskip

We start from a $q$-deformation of the simple random walk,
namely, the \emph{$q$-exchangeable walk} introduced by
Gnedin--Olshanski \cite{Gnedin2009}. Under an
elementary transposition of the walks' increments,
the probability of the trajectory is multiplied by $q$ or $q^{-1}$ (depending on the 
order of the increments),
where $q\in(0,1)$ is a parameter.
When $q=1$, this property reduces to the usual exchangeability.
We show that the condition for independent $q$-exchangeable
random walks never to collide is realized through 
a Doob's $h$-transform for an explicit nonnegative
eigenfunction with eigenvalue $q^{\binom m2}$. 
From this
perspective, our process $\Upsilon_m$ is a $q$-deformation of the
classical model of noncolliding simple random walks
(and, moreover, reduces to this classical model
in a $q\to1$ limit).
Note that all previously studied
noncolliding random walks (including the model of
Borodin--Gorin \cite{BG2011non} where $q$ enters the particle 
speeds $1,q^{-1},q^{-2},\ldots $ and thus
plays a different role) satisfy
the usual, undeformed exchangeability.

\smallskip

Our $q$-dependent process $\Upsilon_m$ is a part of a wider family
of Markov chains with Macdonald parameters $(q,t)$
defined recently by Petrov \cite{petrov2022noncolliding}.
The asymptotics of the latter should be accessible through the 
technique of Gorin--Huang 
\cite{gorin2022dynamical}, but here we stay within the $t=q$
case 
(corresponding to $\beta=2$ in random matrices)
which allows to show local bulk universality.

\smallskip

Let us add that in continuous time and space, noncolliding Brownian motions weighted by the area penalty
and their scaling limit, the 
Dyson Ferrari–Spohn diffusion,
were considered by
Caputo--Ioffe--Wachtel
\cite{CaputoIoffe2019Confinement},
Ferrari--Shlosman
\cite{ferrari2023airy2},
in connection with interfaces in the Ising model in two and three dimensions.

\bigskip

Let us now turn to random lozenge tilings, and provide a
very brief overview of the relevant models and typical
asymptotic results. Random lozenge tilings (equivalently,
	random dimer coverings / perfect matchings on the
hexagonal grid) is a well-studied two-dimensional
statistical mechanical model in which correlations are
expressible as determinants of the inverse Kasteleyn matrix,
as first shown by Kasteleyn \cite{Kasteleyn1967} and
Temperley--Fisher \cite{temperley1961dimer}. See also the
lecture notes by Kenyon \cite{Kenyon2007Lecture} and Gorin
\cite{gorin2021lectures}.
There are several types of asymptotic results about random tilings, 
here we consider the 
\emph{limit shape} and the \emph{bulk} (\emph{lattice}) \emph{universality}.
The latter requires either an
explicit inverse of the Kasteleyn matrix (which heavily depends
on the boundary conditions),
or an effective asymptotic control of this inverse.

\smallskip

The limit shape (law of large numbers) phenomenon states
that the normalized height function of a random tiling
tends
to a nonrandom limiting height function. The latter has a
variational description (Cohn--Kenyon--Propp \cite{CohnKenyonPropp2000} and
Kenyon--Okounkov \cite{OkounkovKenyon2007Limit}). 
In many problems, 
in particular, for uniformly random lozenge tilings 
of the so-called \emph{sawtooth polygons} 
considered in Petrov \cite{Petrov2012}
(see \Cref{fig:sawtooth_limits_intro}),
the variational problem can 
be solved in terms of algebraic equations for the gradient of the 
limiting height function.

\smallskip

In a neighborhood where the limiting height function is non-flat,
the
local (bulk) correlations of a random tiling should be
described by an ergodic translation invariant Gibbs measure
on tilings of the whole plane. 
The latter measure is unique
for a given gradient, and its 
correlations are described by the \emph{incomplete beta kernel} (an
extension of the \emph{discrete sine kernel}); see
Sheffield \cite{Sheffield2008}, Kenyon--Okounkov--Sheffield
\cite{KOS2006}.
This universal bulk behavior has been proven in full
generality for uniformly random tilings by 
Aggarwal \cite{aggarwal2019universality},
following the earlier works
for the hexagon 
(Baik--Kriecherbauer--McLaughlin--Miller \cite{BKMM2003}, 
Gorin \cite{Gorin2007Hexagon}),
sawtooth polygons (Petrov \cite{Petrov2012}), 
and a lozenge tiling model 
corresponding to the noncolliding Bernoulli simple random walks
(Gorin--Petrov \cite{GorinPetrov2016universality}).
The curve separating the region where the height
function is non-flat is referred to as the \emph{frozen boundary}.
For uniformly random tilings of polygons this curve is algebraic.

\smallskip

A $\mathfrak{q}$-deformation of uniformly random lozenge tilings
is obtained by assigning probability weights proportional to 
$\mathfrak{q}^{\mathrm{volume}}$, where the volume is measured under the height function. 
This allows considering tilings of infinite 
domains (also sometimes called plane partitions)
if the partition function is summable for $\mathfrak{q}\in(0,1)$
(we use a different font for $\mathfrak{q}$
to distinguish from the $q$-exchangeable parameter).
Cerf--Kenyon \cite{cerf2001low} studied the $\mathfrak{q}$-weighted
plane partitions and proved a limit shape result.
Okounkov--Reshetikhin \cite{okounkov2003correlation}
found asymptotics of local correlations as $\mathfrak{q}\nearrow1$, 
and
introduced the incomplete beta kernel to describe them.
In the subsequent work 
\cite{Okounkov2005}, they looked at 
random skew plane partitions which may have a back wall.
Depending on the number of turns of the back wall, 
the frozen boundary may form several
asymptotes extending to infinity, see
\cite[Figures 2, 15, 16]{Okounkov2005} for illustrations.

\smallskip

\begin{figure}[htpb]
	\centering
	\includegraphics[width=.35\textwidth]{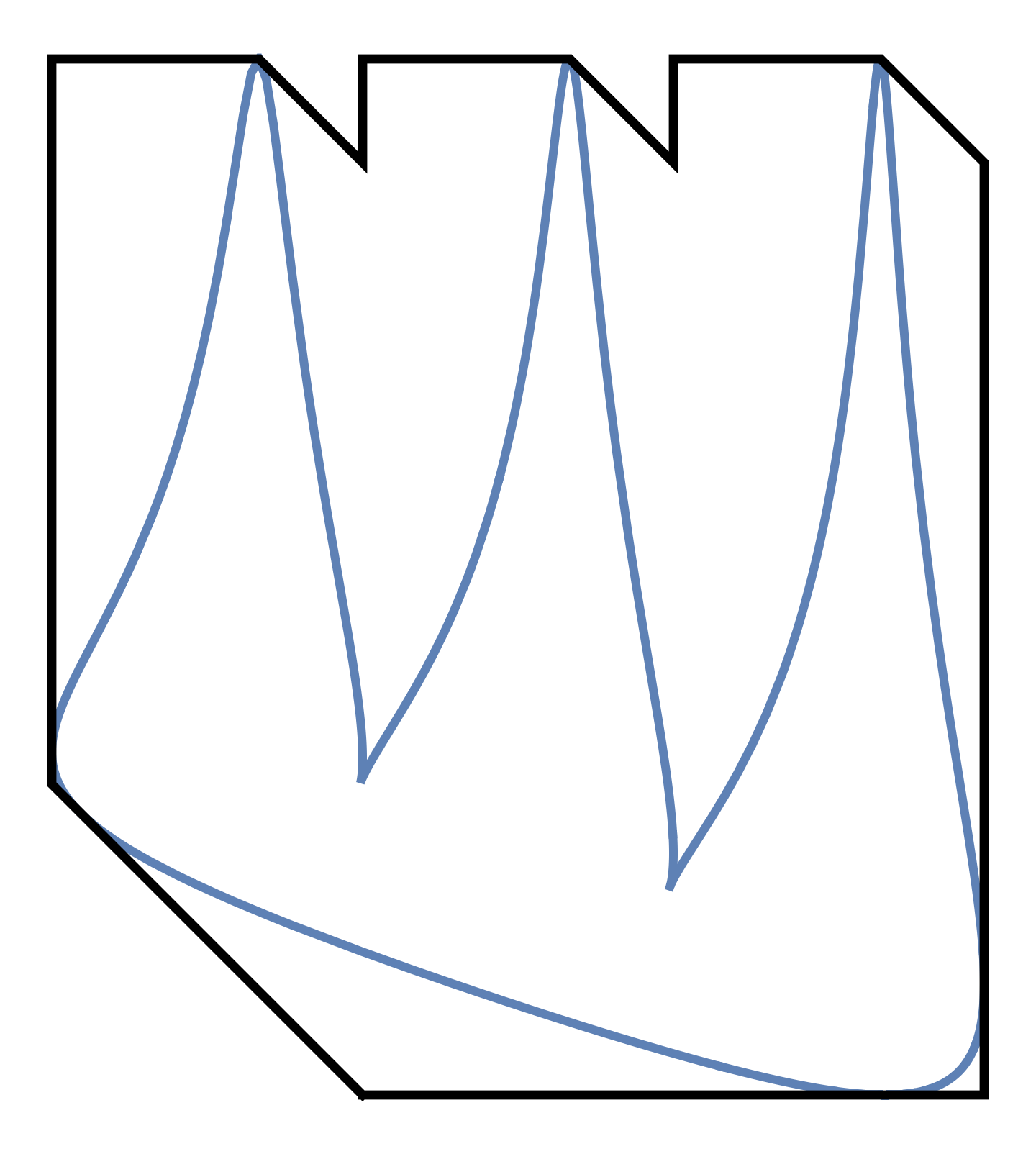}
	\qquad 
	\qquad 
	\includegraphics[width=.35\textwidth]{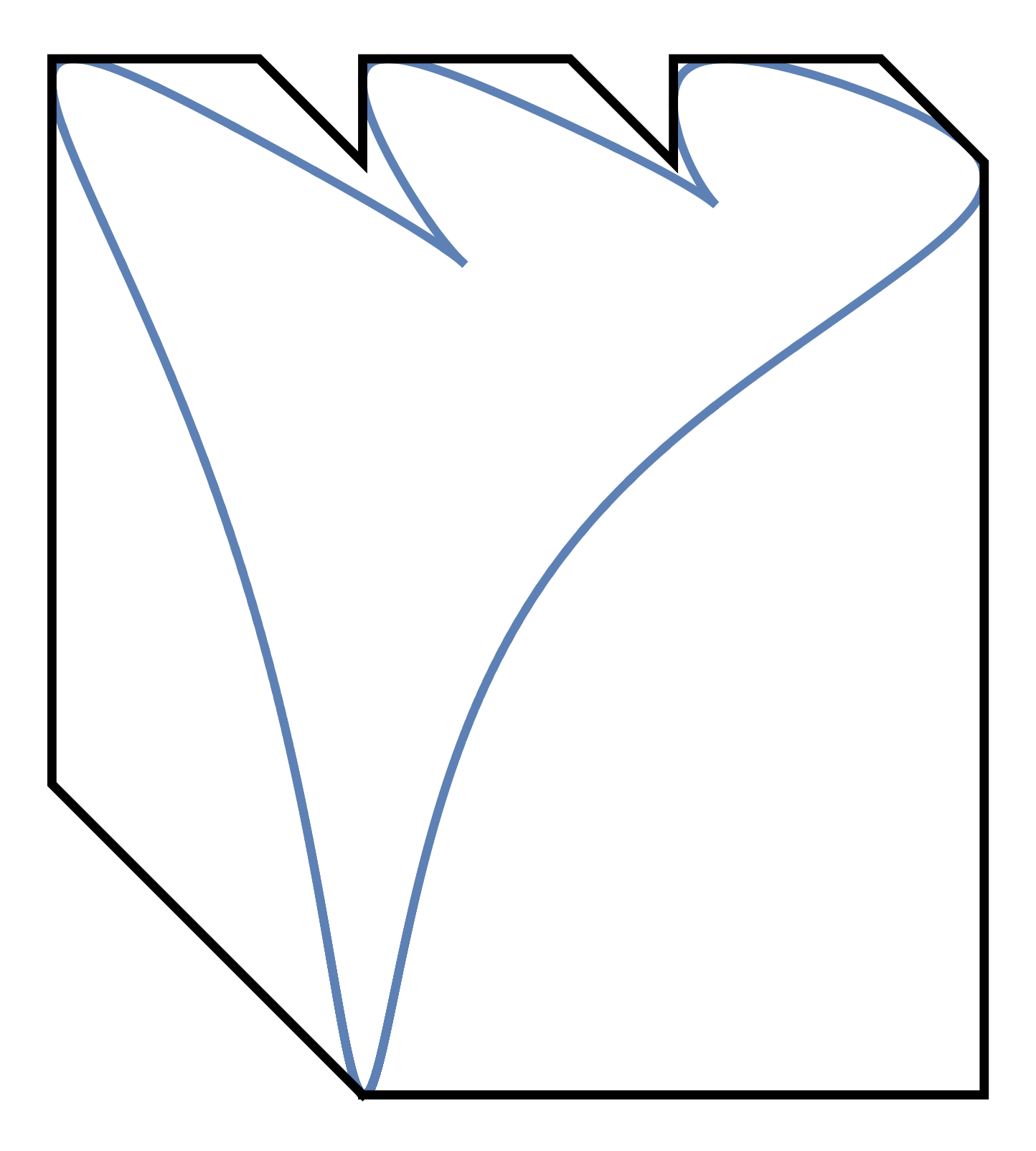}
	\caption{A sawtooth polygon and frozen boundaries
		for the
		$\mathfrak{q}$-weighted lozenge tilings 
		as $\mathfrak{q}=e^{-\upgamma/N}\to1$, where $N\to +\infty$ is the linear
		size of the polygon. Here $\upgamma>0$ on the left and $\upgamma<0$ on the right.
		These frozen boundaries were obtained by 
		Di Francesco--Guitter 
		\cite{DiFran2019qvol}.}
	\label{fig:sawtooth_limits_intro}
\end{figure}

For $\mathfrak{q}$-weighted random tilings of (bounded)
sawtooth polygons, the frozen boundary 
was computed by di Francesco--Guitter
\cite{DiFran2019qvol} using the tangent method.
Gorin--Huang \cite{gorin2022dynamical}
recently 
obtained it 
for an even more general ensemble
with $(\mathfrak{q},\kappa)$-weights
introduced
by Borodin--Gorin--Rains \cite{borodin-gr2009q}.
The boundaries of 
\cite{DiFran2019qvol} turn into the ones of 
\cite{Okounkov2005} in a limit when
the side of the sawtooth polygon 
with multiple defects tends to infinity.
Then
the ``cloud'' parts adjacent to this side
degenerate into multiple asymptotes.

\smallskip 

For sawtooth polygons, bulk universality of the
$\mathfrak{q}$-weighted lozenge tilings as $\mathfrak{q}\to
1$ is still open except for the hexagon case settled by
Borodin--Gorin--Rains \cite{borodin-gr2009q}. This is
despite the explicit double contour integral expression for
the correlation kernel $K_{\mathrm{loz}}$ (a close relative
of the inverse Kasteleyn matrix) given by Petrov
\cite{Petrov2012} which we recall in
\eqref{eq:kernel_from_Pet14} below.  The reason is that the
$\mathfrak{q}$-hypergeometric function under the integral in
$K_{\mathrm{loz}}$ has hindered its direct asymptotic
analysis.

\smallskip

It is known that the $\mathfrak{q}$-dependent kernel
$K_{\mathrm{loz}}$ (and its subsequent asymptotic analysis)
simplifies in several cases. First, setting
$\mathfrak{q}=1$, we get a kernel for uniformly random
tilings, which has led to many asymptotic results; see
Petrov \cite{Petrov2012}, \cite{Petrov2012GFF},
Toninelli--Laslier \cite{laslier2013lozenge}, Gorin--Petrov
\cite{GorinPetrov2016universality}, Aggarwal
\cite{aggarwal2019universality}. In another regime, keeping
$\mathfrak{q}<1$ fixed and sending the top boundary of the
polygon (which has several turns) up to infinity, as in
\Cref{fig:sawtooth_limits_intro}, left, one can show that
the $\mathfrak{q}$-weighted random tilings in a bottom part
of the picture converge to the random plane partitions with
a back wall studied by Okounkov--Reshetikhin
\cite{Okounkov2005}. Moreover, in this limit, the
correlation kernel $K_{\mathrm{loz}}$ turns into the simpler
kernel \cite[(25)]{Okounkov2005} obtained originally via the
technique of Schur processes. The latter kernel is amenable
to asymptotic analysis by the standard steepest descent
method, which in particular leads to the frozen boundary
with several asymptotes (already visible in
\Cref{fig:sawtooth_limits_intro}) and bulk universality.

\bigskip

Our model $\Upsilon_m$ of noncolliding $q$-exchangeable
random walks presents a new case when the complicated kernel
$K_{\mathrm{loz}}$ simplifies. Namely, if instead of
$\mathfrak{q}<1$, we keep $\mathfrak{q}>1$, and send the
\emph{bottom} boundary of the polygon down to infinity as in
\Cref{fig:sawtooth_limits_intro}, right, then around the top
boundary of the polygon we have the convergence to the
noncolliding $q$-exchangeable random walks, with
$q=\mathfrak{q}^{-1}$. The resulting correlation kernel
$K_{\mathrm{walks}}$ has an explicit double contour integral form for
any initial condition in the noncolliding walks. From this
kernel, we obtain the limit shape and bulk universality
results as the number $m$ of walks goes to infinity and
$q=e^{-\gamma/m}\to 1$. The frozen boundary for the
noncolliding walks 
(see \Cref{fig:frozen_boundaries}
for examples)
may form several ``cloud turns'', and always
has exactly one asymptote
(already seen forming in
\Cref{fig:sawtooth_limits_intro}, right).

\smallskip

We conclude that limit shape results for ensembles of random
lozenge tilings are accessible by a variety of methods, but
bulk universality requires knowledge or precise control of the
inverse Kasteleyn matrix (or its close relative, the correlation kernel).
The
$q^{\mathrm{volume}}$ ensemble of random lozenge tilings of sawtooth
polygons is an example of a model where such control 
is still out of reach. In the present paper, we 
explore a new degeneration of this lozenge tiling
model, which is amenable to asymptotic analysis, as well as has a very 
nice interpretation as noncolliding $q$-exchangeable random walks
with arbitrary initial conditions.

\subsection*{Outline}

Above in the Introduction, we gave an overview of 
where our model
$\Upsilon_m$ fits into the classes of noncolliding 
walks and random lozenge tilings.
Below in \Cref{sec:main_results}
we describe our model and results in full detail. 
In particular, we show that our process $\Upsilon_m$
coincides with the system of 
independent $q$-exchangeable random walks 
conditioned never to collide.
The proofs of the determinantal kernel and the asymptotic results
are given in
\Cref{sec:from_tilings_to_walks,sec:limit_in_kernel,sec:asymptotic_analysis}.

\subsection*{Acknowledgments}

LP is grateful to Zhongyang Li for an initial discussion of the problem. The work was partially supported by the NSF grants DMS-1664617 and DMS-2153869, and the Simons Collaboration Grant for Mathematicians 709055. This material is based upon work supported by the National Science Foundation under grant DMS-1928930 while the first author participated in the program ``Universality and Integrability in Random Matrix Theory and Interacting Particle Systems'' hosted by the Mathematical Sciences Research Institute in Berkeley, California, during the Fall 2021 semester.

\section{Model and main results}
\label{sec:main_results}

In this section, we discuss our model of noncolliding $q$-exchangeable random walks and formulate our main results on its determinantal structure and asymptotic behavior.

\subsection{The $q$-exchangeable random walk}
\label{sub:intro_single_walk}

Consider a discrete-time simple random walk $\{y(t)\}_{t\in \mathbb{Z}_{\ge0}}$
on $\mathbb{Z}$ making steps 
$0$ (``straight'') and $-1$ (``down'')
according to independent flips of a given (possibly biased) 
coin.\footnote{It is convenient to have random walks which move down on the one-dimensional 
integer lattice $\mathbb{Z}$.} It is well-known that the sequence of steps in this random walk is 
\emph{exchangeable}, that is,
\begin{equation}
	\label{eq:exchangeable}
	\mathop{\mathrm{Prob}}
	\left( y(t+1)-y(t)=\epsilon_1,\ssp y(t+2)-y(t+1)=\epsilon_2,\ldots
	,y(t+k)-y(t+k-1)=\epsilon_k \right)
\end{equation}
is symmetric in $\epsilon_1,\ldots,\epsilon_k\in\left\{ 0,-1 \right\} $ for any $t\ge0$ and $k\ge1$.

Gnedin--Olshanski
\cite{Gnedin2009}
considered a $q$-deformation of the concept of exchangeability 
depending on a parameter $q\in(0,1)$.
For a \emph{$q$-exchangeable} random walk
$\{y(t)\}_{t\in \mathbb{Z}_{\ge0}}$, the quantity
\eqref{eq:exchangeable} is no longer symmetric in $\epsilon_1,\ldots,\epsilon_k $.
Instead, we have the following $q$-symmetry under elementary transpositions
$\epsilon_i\leftrightarrow \epsilon_{i+1}$:
\begin{multline*}
	\mathop{\mathrm{Prob}}
	\left( \ldots,  y(t+i)-y(t+i-1)=\epsilon_i,\ssp
		y(t+i+1)-y(t+i)=\epsilon_{i+1},\ldots
	\right)
	\\=q^{\epsilon_{i}-\epsilon_{i+1}}
	\mathop{\mathrm{Prob}}
	\left( \ldots,  y(t+i)-y(t+i-1)=\epsilon_{i+1},\ssp
		y(t+i+1)-y(t+i)=\epsilon_{i},\ldots
	\right).
\end{multline*}
In words, under a
transposition of the neighboring steps $
(\textnormal{down},\textnormal{straight}) \to
(\textnormal{straight}, \textnormal{down}) $, the
probability of the trajectory is multiplied by $q$.

\begin{remark}
	\label{rmk:q_convention}
	By convention,
	throughout all exact computations in the paper, the parameter
	$q$ is a fixed number between $0$ and $1$. For the asymptotic analysis of the noncolliding
	$q$-exchangeable random walks, we send $q\nearrow 1$. These asymptotic
	results are formulated
	in \Cref{sub:intro_asymptotics}
	and proven in \Cref{sec:asymptotic_analysis}.
\end{remark}

The space of laws (probability distributions) of
$q$-exchangeable random walks is a convex simplex. That is,
the convex combination (mixture) of probability laws
preserves $q$-exchangeability. By a $q$-analogue of the de
Finetti's theorem proven in \cite{Gnedin2009}, extreme
$q$-exchangeable random walks (that is, extreme points of
this simplex) are parametrized by $\Delta\coloneqq \left\{
0,1,2,\ldots  \right\}\cup\left\{ \infty \right\}$. In
detail, for any $q$-exchangeable random walk $y(t)$, there
exists a probability measure $\mu$ on $\Delta$ such that the
law of $y(t)$ is a mixture of the extreme distributions by
means of $\mu$.

\medskip

Our first observation is that all extreme $q$-exchangeable
random walks parametrized by points of
$\Delta_{\mathrm{fin}}\coloneqq\left\{ 0,1,2,\ldots
\right\}\subset \Delta$ are \emph{one and the same
space-inhomogeneous random walk} with varying initial
configuration and an absorbing wall at $0$.

\begin{definition}[The $q$-exchangeable random walk in $\mathbb{Z}_{\ge0}$]
	\label{def:q_exch_walk}
	Let $\Upsilon_1$ be the following one-step Markov transition probability, where
	$x,y\in \mathbb{Z}_{\ge0}$:
	\begin{equation*}
		\Upsilon_1(x,y)
		\coloneqq
		\begin{cases}
			q^x,&y=x;\\
			1-q^x,&y=x-1;\\
			0,&\textnormal{otherwise}.
		\end{cases}
	\end{equation*}
\end{definition}

\begin{proposition}
	\label{prop:q_exch}
	\begin{enumerate}[\bf1.\/]
		\item 
			Started from any $x\in \mathbb{Z}_{\ge0}$,
			the random walk with transition probabilities $\Upsilon_1$ is $q$-exchangeable.
		\item 
			Any extreme $q$-exchangeable random walk 
			parametrized by a point $x\in \Delta_{\mathrm{fin}}$
			can be identified with the random walk $\Upsilon_1$ started from $x$.
	\end{enumerate}
\end{proposition}
\begin{proof}
	For the first part, we have
	\begin{equation*}
		\Upsilon_1(x,x)\Upsilon_1(x,x-1)=q^x(1-q^x)=q\cdot \Upsilon_1(x,x-1)\Upsilon_1(x-1,x-1),
	\end{equation*}
	which immediately implies the $q$-exchangeability.
	The second part follows by comparing our random walk with the one described in 
	\cite[Proposition 4.1]{Gnedin2009}.
\end{proof}

\begin{remark}
	\label{rmk:limit_q_to_1}
	In the scaling limit as $q=e^{-\varepsilon}\to1$ and $x=\lfloor \varepsilon^{-1}\log(1/p) \rfloor +\tilde x$,
	where $p\in(0,1)$ and $\tilde x\in \mathbb{Z}$,
	the $q$-exchangeable random walk on $\mathbb{Z}_{\ge0}$
	turns into the usual simple random walk on $\mathbb{Z}$.
	The latter corresponds to independent coin flips with the probability of Heads equal to $p$.
\end{remark}

\subsection{Noncolliding simple random walks}
\label{sub:intro_model_noncolliding}

The central object of the present paper is a model of
several interacting $q$-exchangeable random walks which never collide. 
Here we first discuss the well-known
model of noncolliding simple random walks on $\mathbb{Z}$.
Thanks to \Cref{rmk:limit_q_to_1}, this well-known model may be viewed as a 
$q\to1$ limit of our model of noncolliding $q$-exchangeable random walks.
We define the latter in detail in \Cref{sub:q_noncoll_intro} below.

The model of noncolliding simple random walks on $\mathbb{Z}$
dates back to Karlin--McGregor \cite{KMG59-Coincidence}.
The model of noncolliding Brownian motions on $\mathbb{R}$
is the celebrated Dyson Brownian motion for the Gaussian Unitary Ensemble
\cite{dyson1962brownian}.
A systematic treatment of noncolliding random walks
connecting them to determinantal
point processes (in particular, orthogonal polynomial ensembles)
is performed by
K{}\"onig--O’Connell--Roch
\cite{konig2002non}.

The one-step Markov transition probability of a model of $m$ independent discrete-time
simple random walks on $\mathbb{Z}$ (making steps $0$ and $-1$ with probabilities $p$ and $1-p$)
conditioned never to collide has the form of a Doob's $h$-transform
\cite[2.VI.13]{doob1984classical}
\begin{equation}
	\label{eq:h_transform_q_is_1_intro}
	\Upsilon_m^{(q=1)}(\vec x, \vec y)=
	\frac{h_m^{(q=1)}(\vec y)}{h_m^{(q=1)}(\vec x)}\ssp 
	\underbrace{\prod_{i=1}^{m}
	\big( 
		p\ssp \mathbf{1}_{y_i=x_i}
		+
		(1-p)\ssp \mathbf{1}_{y_i=x_i-1}
\big)}_{\Upsilon_{m,\mathrm{ind}}^{(q=1)}(\vec x,\vec y)},
\end{equation}
where $\vec x=(x_1>\ldots>x_m )$, $\vec y=(y_1>\ldots>y_m )$, $x_i,y_i\in \mathbb{Z}$.
Here 
\begin{equation}
	\label{eq:Vandermonde_q_is_1}
	h_m^{(q=1)}(\vec x)\coloneqq \prod_{1\le i<j\le m}(x_i-x_j)
\end{equation}
is the Vandermonde determinant, and the product over $i$ in \eqref{eq:h_transform_q_is_1_intro} is simply the one-step Markov transition probability of a collection of $m$ independent simple random walks, which we denoted by 
$\Upsilon_{m,\mathrm{ind}}^{(q=1)}$. In \eqref{eq:h_transform_q_is_1_intro} and throughout the text, $\mathbf{1}_{A}$ stands for the indicator of an event or a condition~$A$.

The fact that \eqref{eq:h_transform_q_is_1_intro} defines a random walk of $m$ particles is not straightforward. 
The key property is that the right-hand side sums to $1$ over all $\vec y$.
Equivalently, $h_m^{(q=1)}(\vec x)$ is a nonnegative harmonic function 
for the collection of $m$ independent simple random walks:
\begin{equation}
	\label{eq:harmonic_q_is_1}
	\sum_{\vec y}
	h_m^{(q=1)}(\vec y)\ssp
	\Upsilon_{m,\mathrm{ind}}^{(q=1)}(\vec x,\vec y)
	= 
	h_m^{(q=1)}(\vec x).
\end{equation}

\subsection{Noncolliding $q$-exchangeable random walks}
\label{sub:q_noncoll_intro}

Here we describe our main model $\Upsilon_m$, which is a
$q$-deformation of the classical model of noncolliding
simple random walks from
\Cref{sub:intro_model_noncolliding}.  For $m=1$, the model
$\Upsilon_1$ is the $q$-exchangeable random walk from
\Cref{sub:intro_single_walk} above.

Denote by $\mathbb{W}_m$ the space of $m$-particle configurations
in $\mathbb{Z}_{\ge0}$:
\begin{equation}
	\label{eq:W_m_space}
	\mathbb{W}_m\coloneqq
	\{
		\vec{x}=(x_1>x_2>\ldots>x_m\ge0 )
	\}\subset \mathbb{Z}_{\ge0}^{m},
\end{equation}
and set $|\vec x|\coloneqq x_1+\ldots+x_m$.
\begin{definition}
	\label{def:Upsilon_m}
	We consider a Markov chain $\Upsilon_m$ on $\mathbb{W}_m$ with the following one-step transition probabilities,
	where 
	$\vec x,\vec y\in \mathbb{W}_m$:
	\begin{equation}
		\label{eq:intro_Upsilon}
		\Upsilon_m(\vec x,\vec y)= 
		q^{-\binom m2+(m-1)\left( |\vec x|-|\vec y| \right)}
		\prod_{1\le i<j\le m}\frac{q^{y_j}-q^{y_i}}{q^{x_j}-q^{x_i}}
			\underbrace{\prod_{i=1}^{m}
			\bigl( q^{x_i}\mathbf{1}_{y_i=x_i}+
		(1-q^{x_i})\mathbf{1}_{y_i=x_i-1}\bigr)}_{\Upsilon_{m,\mathrm{ind}}(\vec x,\vec y)}.
	\end{equation}
	See \Cref{fig:walks_intro}, left, for an illustration of the trajectory
	of the process $\Upsilon_m$ (with $m=4$) started from $\vec x=(7,6,3,1)$.
	As time goes to infinity, the dynamics $\Upsilon_m$ reaches its unique absorbing state
	$\delta_m\coloneqq (m-1,m-2,\ldots,1,0 )\in \mathbb{W}_m$.
	We call the Markov process $\Upsilon_m$ the \emph{noncolliding $q$-exchangeable random walks}.
		
\end{definition}

The process $\Upsilon_m$ was introduced recently by Petrov \cite{petrov2022noncolliding}.
It is a particular $t=q$ case of the Macdonald noncolliding random walks, and 
the main goal of the present paper is a detailed asymptotic investigation of the 
$t=q$ case. It turns out that the $t=q$ process is determinantal with an explicit 
double contour integral kernel.
The asymptotic analysis of the general $(q,t)$ Macdonald 
case should be performed by different, non-determinantal methods (potentially 
based on \cite{gorin2022dynamical}), but
we leave this question out of the scope of the present work.

From the results of
\cite{petrov2022noncolliding} it follows that 
for any $\vec x\in \mathbb{W}_m$, the quantities $\Upsilon_m(\vec x,\vec y)$ are nonnegative and sum to 
$1$ over all $\vec y\in \mathbb{W}_m$. Equivalently, 
the following $q$-deformation of the Vandermonde determinant
\begin{equation}
	\label{eq:q_Vandermonde}
	h_m(\vec x)\coloneqq q^{-(m-1)|\vec x|}\prod_{1\le i<j\le m}(q^{x_j}-q^{x_i})
\end{equation}
is a nonnegative eigenfunction for 
$\Upsilon_{m,\mathrm{ind}}$, 
the collection of $m$ independent 
$q$-exchangeable random walks:
\begin{equation}
	\label{eq:h_transform_for_q}
	\sum_{\vec y\in \mathbb{W}^m}
	h_m(\vec y)\ssp
	\Upsilon_{m,\mathrm{ind}}(\vec x,\vec y)
	= 
	q^{\binom m2}
	h_m(\vec x).
\end{equation}
This property is similar to \eqref{eq:harmonic_q_is_1}, but
observe that here the function $h_m$ is not harmonic with
eigenvalue $1$, but instead, its eigenvalue is equal to
$q^{\binom {m}{2}}$.

From \eqref{eq:intro_Upsilon}--\eqref{eq:h_transform_for_q}
we see that the transition probabilities of the dynamics $\Upsilon_m$ 
have a Doob's $h$-transform like form.
Moreover, similarly to the simple random walk case in \Cref{sub:intro_model_noncolliding},
our process
$\Upsilon_m$ can be obtained from 
the independent $q$-exchangeable walks
$\Upsilon_{m,\mathrm{ind}}$ 
by conditioning them never to collide.
To formulate this result (\Cref{prop:never_to_collide_q} below), we need some notation.
Denote 
by $\Upsilon_1^{(T)}$ the $T$-step transition probability
of the single $q$-exchangeable random walk.
One can readily compute this probability assuming 
that $T\ge x-y$:\footnote{Here and throughout
the paper we use the $q$-Pochhammer symbols notation
\begin{equation*}
	(a ; q)_k\coloneqq (1-a)(1-a q) \ldots(1-a q^{k-1}), \quad k \in \mathbb{Z}_{\geq 0},
\end{equation*}
and $(z ; q)_{\infty}\coloneqq \prod_{i=0}^{\infty}\left(1-z q^i\right)$ is a convergent infinite product because $q\in(0,1)$.
The last ratio in \eqref{eq:Upsilon_1_T} is the 
$q$-binomial coefficient since $\binom{n}{k}_q = \frac{(q;q)_n}{(q;q)_k(q;q)_{n-k}}$.}
\begin{equation}
	\label{eq:Upsilon_1_T}
	\Upsilon_1^{(T)}(x,y)=\mathbf{1}_{0\le y\le x}\ssp 
	(1-q^x)(1-q^{x-1})\ldots(1-q^{y+1}) 
	\ssp
	q^{y(T-x+y)}
	\ssp
	\frac{(q;q)_T}{(q;q)_{x-y}(q;q)_{T-x+y}}.
\end{equation}
Indeed, 
$(1-q^x)(1-q^{x-1})\ldots(1-q^{y+1}) 
\ssp
q^{y(T-x+y)}$ is the probability that the
walk first goes all the way down from $x$ to $y$ and then stays at $y$, 
and the $q$-binomial coefficient 
$\frac{(q;q)_T}{(q;q)_{x-y}(q;q)_{T-x+y}}$ comes from the $q$-exchangeability.

By \cite{KMG59-Coincidence}, the $T$-step transition probability of an $m$-particle
independent $q$-exchangeable random walk $\Upsilon_{m,\mathrm{ind}}$ conditioned on the event that the particles do not collide 
over these $T$ steps is equal to 
$\det[\Upsilon_1^{(T)}(x_i,y_j)]_{i,j=1}^m$,
where $\vec x,\vec y\in \mathbb{W}_m$.
Therefore, the one-step transition probability from $\vec x$ to $\vec y$
of $\Upsilon_{m,\mathrm{ind}}$ conditioned to not collide up to time $T$ 
and to get absorbed at $\delta_m$ has the form
\begin{equation*}
	\frac{\det[\Upsilon_1^{(T-1)}(y_i,m-j)]_{i,j=1}^m}{\det[\Upsilon_1^{(T)}(x_i,m-j)]_{i,j=1}^m}
	\ssp
	\Upsilon_{m,\mathrm{ind}}(\vec x,\vec y)
	\ssp
	\mathbf{1}_{y_i-x_i \in \left\{ 0, -1 \right\}\textnormal{ for all $i$}}.
\end{equation*}

\begin{proposition}
	\label{prop:never_to_collide_q}
	For any $\vec x,\vec y\in \mathbb{W}_m$, we have
	\begin{equation}
		\label{eq:never_to_collide_limit}
		\lim_{T\to+\infty}
		\frac{\det[\Upsilon_1^{(T-1)}(y_i,m-j)]_{i,j=1}^m}{\det[\Upsilon_1^{(T)}(x_i,m-j)]_{i,j=1}^m}
		=
		q^{-\binom m2+(m-1)\left( |\vec x|-|\vec y| \right)}
		\prod_{1\le i<j\le m}\frac{q^{y_j}-q^{y_i}}{q^{x_j}-q^{x_i}}.
	\end{equation}
\end{proposition}
\begin{proof}
	Using \eqref{eq:Upsilon_1_T} and factoring out $q^{\sum_{j=1}^m(m-j)(T-1+m-j)}$ and $q^{\sum_{j=1}^m(m-j)(T+m-j)}$, respectively, from the numerator and the denominator in the right-hand side \eqref{eq:never_to_collide_limit} yields the factor $q^{-\binom{m}{2}}$ in the left-hand side. After this, we can pass to the limit as $T\to\infty$ in the matrix elements of each determinant because the resulting matrices stay nondegenerate. Thus, it remains to compute one such determinant with $T=\infty$, say (after replacing the index $j$ with $m+1-j$):
	\begin{equation}
		\label{eq:never_to_collide_limit_proof}
		\det\Bigl[ 
			\mathbf{1}_{0\le j-1\le x_i}\ssp 
			\frac{
			(1-q^{x_i})(1-q^{x_i-1})\ldots(1-q^{j}) 
			\ssp
			q^{-x_i(j-1)}}{(q;q)_{x_i-j+1}}
		\Bigr]_{i,j=1}^{m}
		=
		\det\Bigl[ 
			\frac{
			q^{-x_i(j-1)}(q^{x_i-j+2};q)_{j-1}}{(q;q)_{j-1}}
		\Bigr]_{i,j=1}^m.
	\end{equation}
	Here we rewrote the products in a convenient form, and observed that the 
	indicator 
	$\mathbf{1}_{0\le j-1\le x_i}$
	is automatically enforced in the right-hand side by the $q$-Pochhammer
	$(q^{x_i-m+j+1};q)_{j-1}$.

	We see that each $(i,j)$-th entry in the determinant in the right-hand side of
	\eqref{eq:never_to_collide_limit_proof} is a polynomial in $q^{-x_i}$ of degree $j-1$.
	Therefore, the whole determinant is proportional to the Vandermonde 
	$\prod_{1\le i<j\le m}(q^{x_j}-q^{x_i})$. 
	One readily sees that the coefficient by this Vandermonde
	is equal to 
	$\frac{(-1)^{m(m-1)/2}q^{-(m-1)|\vec x|}}{(q;q)_1(q;q)_2\ldots(q;q)_{m-1}}$,
	which completes the proof.
\end{proof}

The limit relation in \Cref{prop:never_to_collide_q}
completes the analogy between the 
well-known model of noncolliding simple random walks on $\mathbb{Z}$
(and the Dyson Brownian motion) and our noncolliding $q$-exchangeable random
walks $\Upsilon_m$.
In both cases, the $h$-transform structure of the 
transition probability is due to the conditioning that the independent random
walks never collide.

\subsection{Gibbs interpretation as $q$-weighted lozenge tilings}
\label{sub:gibbs_q_lozenge_partition_function}

The $m$-particle process $\Upsilon_m$ satisfies a 
version of the $q$-exchangeability 
discussed for a single random walk in \Cref{sub:intro_single_walk}. 
Namely, this is the Gibbs property of the walk observed in
\cite{petrov2022noncolliding}.

Fix $m$ and an initial condition $\vec x\in \mathbb{W}_{m}$
for the process $\Upsilon_m$. Under a suitable affine
transformation of the trajectory of $\Upsilon_m$, it can be
bijectively identified with a lozenge tiling of the vertical
strip of width $x_1+1$, see \Cref{fig:walks_intro}, right.
The bottom boundary of the vertical strip is encoded by
$\vec{x}$ in the following way. Viewing $\vec{x}$ as a
particle configuration in $\mathbb{Z}_{\ge0}$, each particle
$x_i$ corresponds to a straight piece in the boundary of
slope $(-1/\sqrt3)$, and each hole in $\vec x$ corresponds
to cutting a small triangle out of the strip. Due to the
eventual absorption of the walk $\Upsilon_m$ at $\delta_m$,
the lozenge tiling is ``frozen'' far at the top, with
$x_1+1-m$ tiles of one type on the left followed by $m$
tiles of the other type. Thus, each lozenge tiling 
corresponding to a trajectory of $\Upsilon_m$ contains only finitely many horizontal lozenges.

\begin{figure}[htpb]
	\centering
	\includegraphics[width=.9\textwidth]{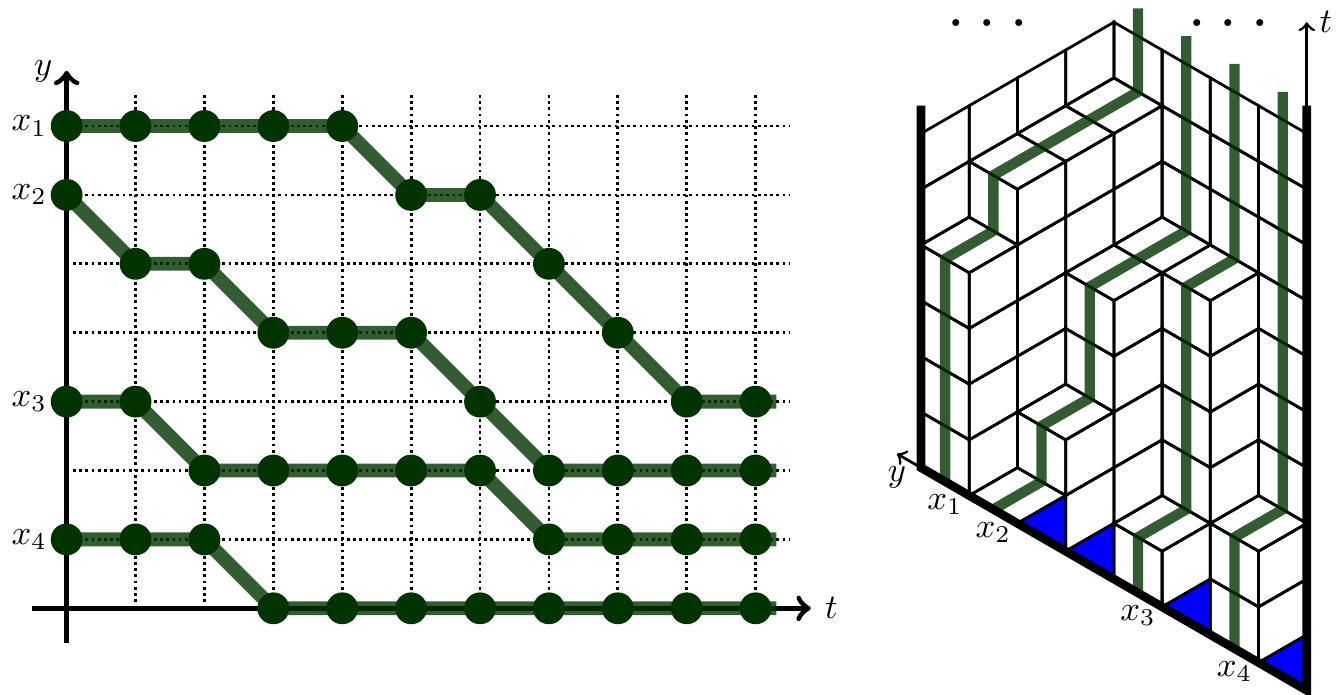}
	\caption{
		Left:
		Illustration of the trajectory
		of the noncolliding $q$-exchangeable random walks
		$\Upsilon_m$ (with $m=4$) started from $\vec x=(7,6,3,1)$.
		Right: A bijective interpretation of the trajectory
		as a lozenge tiling of a strip via an affine transformation.}
	\label{fig:walks_intro}
\end{figure}

The lozenge tiling corresponding to a trajectory of $\Upsilon_m$ 
can be interpreted as a \emph{stepped surface} in three dimensions
such that the solid under this surface is made
out of $1\times 1\times 1$ boxes. 
Via this interpretation, each trajectory of $\Upsilon_m$
has a well-defined \emph{volume} under the corresponding stepped surface.
In other words, the volume of a given tiling is the number of boxes which must be added to the 
minimal configuration to get this tiling. For example, 
the volume of the tiling in \Cref{fig:walks_intro}, right, is equal to~34.

The next statement follows from \cite[Proposition 10]{petrov2022noncolliding}.

\begin{proposition}
	\label{prop:q_vol_Upsilon_identification}
	Fix $m$ and $\vec x\in \mathbb{W}_m$.
	The probability distribution of the trajectory of the Markov process
	$\Upsilon_m$ \eqref{eq:intro_Upsilon} started from $\vec x$ 
	is the same as the distribution of the random lozenge tiling
	of the strip as in \Cref{fig:walks_intro}, right (depending on $\vec x$), 
	where the probability weight of a tiling is proportional
	to $q^{\mathrm{volume}}$.
\end{proposition}

\begin{remark}
	\label{rmk:compare_with_BG2011}
	Another $q$-dependent model
	of noncolliding random walks was 
	introduced and studied in
	\cite{BG2011non}. 
	In that model, the parameter $q$ 
	enters the particle speeds $1,q^{-1},q^{-2},\ldots $,
	but the dynamics as a whole satisfies the usual, undeformed
	exchangeability property. Indeed,
	the single-particle dynamics in \cite{BG2011non} is the 
	simple random walk (with Poisson, Bernoulli, or geometric jumps),
	and our single-particle dynamics is the $q$-exchangeable random walk.
\end{remark}

Let us denote the $q^{\mathrm{volume}}$-weighted probability measure on tilings
described in \Cref{prop:q_vol_Upsilon_identification} by $\mathscr{M}_m^{(\vec x)}$.
The partition function (that is, the probability normalizing constant)
of $\mathscr{M}_m^{(\vec x)}$
has an explicit form:

\begin{proposition}
	\label{prop:q_vol_Upsilon_partition_function}
	The sum of the quantities $q^{\mathrm{volume}}$ over all 
	lozenge tilings of the strip as in \Cref{fig:walks_intro}, right
	(determined by $\vec x\in \mathbb{W}_m$), is equal to 
	\begin{equation}
		\label{eq:Z_m_partition_function}
		Z[ \mathscr{M}_m^{(\vec x)} ]=
		\prod_{i=1}^{m}\frac{1}{(q;q)_{x_i}}
		\prod_{1\le i<j\le m}(1-q^{x_i-x_j}).
	\end{equation}
\end{proposition}
Note in particular that setting $q=0$, we get $Z[ \mathscr{M}_m^{(\vec x)} ]= 1$, as it should be.
\begin{proof}[Proof of \Cref{prop:q_vol_Upsilon_partition_function}]
	From \Cref{prop:q_vol_Upsilon_identification},
	it suffices to check that $1/Z[ \mathscr{M}_m^{(\vec x)} ]$ is 
	the transition probability (over $x_1-m+1$ steps) of fastest path 
	from $\vec x$ to the absorbing state $\delta_m$.
	This fact follows by taking the
	product of the one-step transition probabilities 
	\eqref{eq:intro_Upsilon} and 
	observing that the $q$-Vandermonde factors
	cancel out, except for the first such factor 
	coming from the initial condition $\vec x$.
	The resulting expression for 
	$1/Z[ \mathscr{M}_m^{(\vec x)} ]$ is then verified directly.
\end{proof}

\subsection{Determinantal kernel}
\label{sub:det_kernel}

Fix $m\in \mathbb{Z}_{\ge1}$ and an initial configuration
$\vec x\in \mathbb{W}_m$ for the 
noncolliding $q$-exchangeable random walks
$\Upsilon_m$
(\Cref{def:Upsilon_m}).
View the trajectory 
$\vec y(t)$
of the process $\Upsilon_m$ as a 
random point configuration 
$\{y_j(t)\colon  j=1,\ldots,m,\,t\in \mathbb{Z}_{\ge0} \}
\subset
\mathbb{Z}^2_{\ge0}$.
The next statement is our first main result.

\begin{theorem}
	\label{thm:det_kernel_for_walks_intro}
	Thus defined random point configuration in $\mathbb{Z}_{\ge0}^{2}$
	forms a determinantal point process,
	that is, for any $\ell \ge1$ and any
	pairwise distinct points $(u_i,t_i)\in \mathbb{Z}_{\ge0}^2$,
	we have
	\begin{equation}
		\label{eq:det_def_proc_intro}
		\begin{split}
			&\mathrm{Prob}
			\left( \textnormal{the random configuration $\{y_j(t)\colon  1\le j\le m,\,t\ge0\}$ contains all
				$(u_i,t_i)$, $i=1,\ldots,\ell $}
			\right)
			\\&\hspace{280pt}=
			\det\bigl[ K_{\mathrm{walks}}(u_i,t_i;u_j,t_j) \bigr]_{i,j=1}^{\ell},
		\end{split}
	\end{equation}
	where the correlation kernel
	$K_{\mathrm{walks}}$ 
	has double contour integral form:
	\begin{equation}
		\label{eq:K_walks_intro}
		\begin{split}
			&K_{\mathrm{walks}}(y_1,t_1;y_2,t_2)
			\coloneqq
			\mathbf{1}_{t_1=t_2}\mathbf{1}_{y_1=y_2}-
			\mathbf{1}_{t_2>t_1}\mathbf{1}_{y_2+t_2>y_1+t_1}
			\frac{q^{(t_1-t_2)(y_1+t_1)}(q^{y_1-y_2+t_1-t_2+1};q)_{t_2-t_1-1}}{(q;q)_{t_2-t_1-1}}
			\\&\hspace{20pt}
			-\frac{q^{-t_1-y_1}}{(2\pi\mathbf{i})^2}
			\oiint
			dz\ssp dw\ssp
			\frac{z^{-t_2}w^{t_1} }{w-z}
			\frac{(q;q)_{t_1}}{(wq^{-y_1-t_1};q)_{t_1+1}}
			\frac{(zq^{1-y_2-t_2};q)_{t_2-1}}{(q;q)_{t_2-1}}
			\frac{(w^{-1};q)_{\infty}}{(z^{-1};q)_{\infty}}
			\prod_{r=1}^{m}\frac{1-q^{x_r}/z}{1-q^{x_r}/w}.
		\end{split}
	\end{equation}
	Here $y_1,y_2\in \mathbb{Z}$, $t_1\in \mathbb{Z}_{\ge0}$, $t_2\in \mathbb{Z}_{>0}$,
	the $w$ contour is an arbitrarily small 
	circle around $0$, and the $z$ contour goes around $q^{y_2+t_2},q^{y_2+t_2+1},q^{y_2+t_2+2},\ldots $, $0$, the $w$ contour, and 
	encircles
	no other $z$ poles of the integrand.
\end{theorem}

We prove \Cref{thm:det_kernel_for_walks_intro}
in \Cref{sec:limit_in_kernel} below after relating (in \Cref{sec:from_tilings_to_walks}) the 
process $\Upsilon_m$ of noncolliding $q$-exchangeable random 
walks to a $q$-weighted distribution on lozenge tilings of a sawtooth polygons.
The determinantal kernel for the latter is known from \cite{Petrov2012}.

\subsection{Asymptotic results}
\label{sub:intro_asymptotics}

Recall the definition of a determinantal kernel which
should appear in the bulk of our noncolliding $q$-exchangeable random
walks as the number of walks and the time go to infinity:

\begin{definition}
	\label{def:incomplete_beta_kernel}
	Let $\omega\in\mathbb{C}\setminus\{0,1\}$, $\Im\omega\ge0$, be a parameter called the \emph{complex slope}. The \emph{incomplete beta kernel} is defined as
  \begin{equation}
		\label{eq:incomplete_beta_formula}
		\mathsf{B}_{\omega}(\Delta t,\Delta p):=
    \frac{1}{2\pi\i}
    \int_{\overline\omega}^{\omega}
    (1-u)^{\Delta t}u^{-\Delta p-1}du,
    \qquad
		\Delta t,\Delta p\in\mathbb{Z},
  \end{equation}
	where the integration arc from $\overline\omega$ to $\omega$
	crosses $(0,1)$ for $\Delta t\ge0$ and
	$(-\infty,0)$ for $\Delta t<0$.
\end{definition}

The kernel \eqref{eq:incomplete_beta_formula} was
introduced in
\cite{okounkov2003correlation} to describe local
asymptotics of a certain ensemble 
of $q$-distributed
random lozenge tilings of the whole plane (equivalent to 
random plane partitions).
Moreover, the incomplete beta kernel is the 
universal bulk scaling limit
of uniformly random lozenge tilings of bounded shapes
\cite{aggarwal2019universality}.
By
\cite{Sheffield2008}, \cite{KOS2006}, for every complex
slope $\omega$, there is a unique ergodic translation
invariant Gibbs measure on lozenge tilings of the whole plane, and
its determinantal correlation kernel is
$\mathsf{B}_\omega$.

\medskip

Let us now describe the asymptotic regime of our random walks.
Let $m\to +\infty$, and set $q=e^{-\gamma/m}$ for fixed $\gamma>0$.
Scale the time and the space variables in 
the random walk (as in \Cref{fig:walks_intro}, left)
proportionally to $m$: $t=\lfloor \tau m \rfloor $,
$y=\lfloor \rho m \rfloor $, where $\tau,\rho\in \mathbb{R}_{\ge0}$ are fixed.
Let the initial condition $\vec x \in \mathbb{W}_m$ 
form a fixed number $L\ge 1$ of densely packed clusters:
\begin{equation}
	\label{eq:densely_packed_s2}
	x_i=i+\lfloor m\ssp C_k \rfloor \quad \textnormal{for $m\ssp a_k\le i< m\ssp a_{k+1}$},
\end{equation}
where $i=1,\ldots,m$, $k=1,\ldots,L $, and
$0<C_1<C_2<\ldots<C_L $ and $0=a_1<a_2<\ldots<a_{L+1}=1 $
are the fixed parameters of the clusters. See
the beginning of \Cref{sub:setup_asymptotics} for more detail.

In the $(\tau,\rho)$ plane, 
let $\partial\mathcal{D}$ be the curve
with the following rational parametrization in the 
exponential coordinates 
$(e^{\gamma\tau},e^{\gamma\rho})$:
\begin{equation*}
	e^{\gamma\tau(w)}=
	\frac{(w F(w))'-e^{-\gamma}}
	{w F'(w)-F(w)+e^{\gamma } F^2(w)}
	,\qquad 
	e^{\gamma\rho(w)}=
	\frac{e^{\gamma } F'(w)}{e^{\gamma } (w F(w))'-1},
	\qquad w\in \mathbb{R},
\end{equation*}
where 
\begin{equation*}
	F(w)\coloneqq
	\frac{w}{w-1}
	\prod_{i=1}^{L}\frac{w \ssp e^{\gamma(a_i+C_i)}-1}{w\ssp e^{\gamma(a_{i+1}+C_i)}-1}
	.
\end{equation*}
This curve bounds a domain denoted by $\mathcal{D}$
such that $\mathcal{D}\cap ([0,\tau]\times \mathbb{R}_{\ge0})$
is bounded for any $\tau>0$.
We call $\partial\mathcal{D}$ the \emph{frozen boundary curve}, 
and $\mathcal{D}$ the \emph{liquid region}.
See 
\Cref{fig:frozen_boundaries} for examples.

\begin{figure}[htpb]
	\centering
	\includegraphics[width=.32\textwidth,trim={0 0 2cm 0},clip]{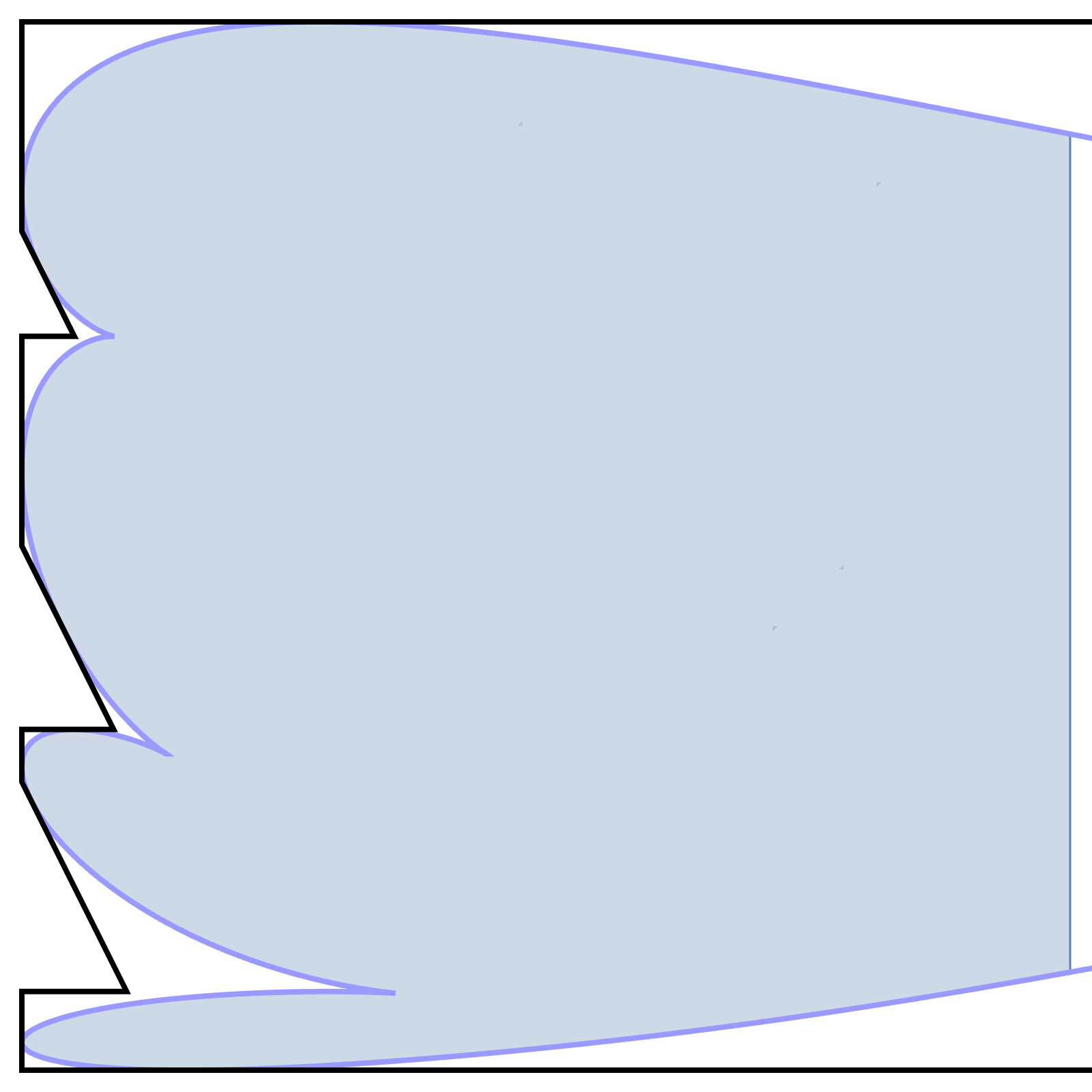}\hspace{.01\textwidth}
	\includegraphics[width=.32\textwidth,trim={0 0 2cm 0},clip]{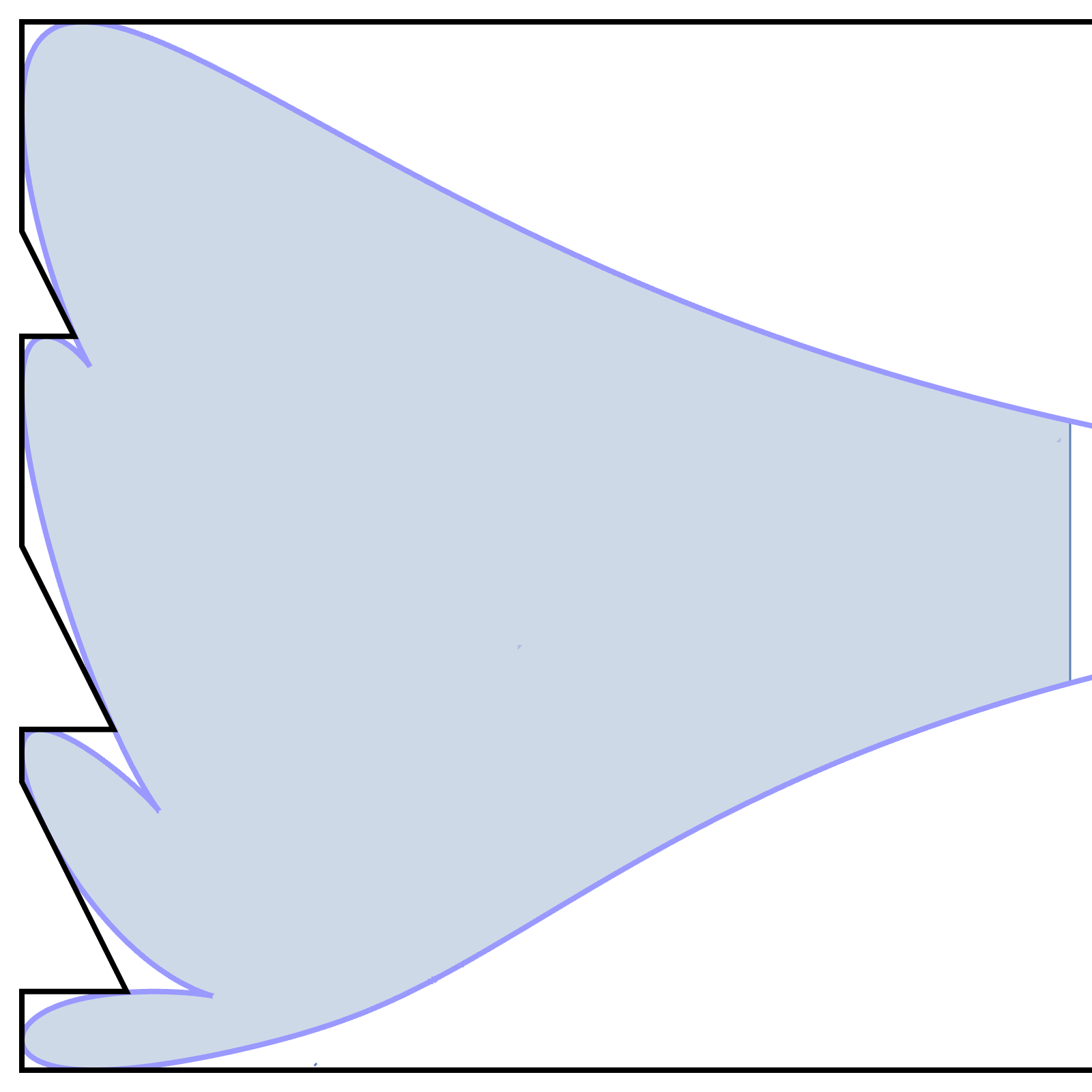}\hspace{.01\textwidth}
	\includegraphics[width=.32\textwidth,trim={0 0 2cm 0},clip]{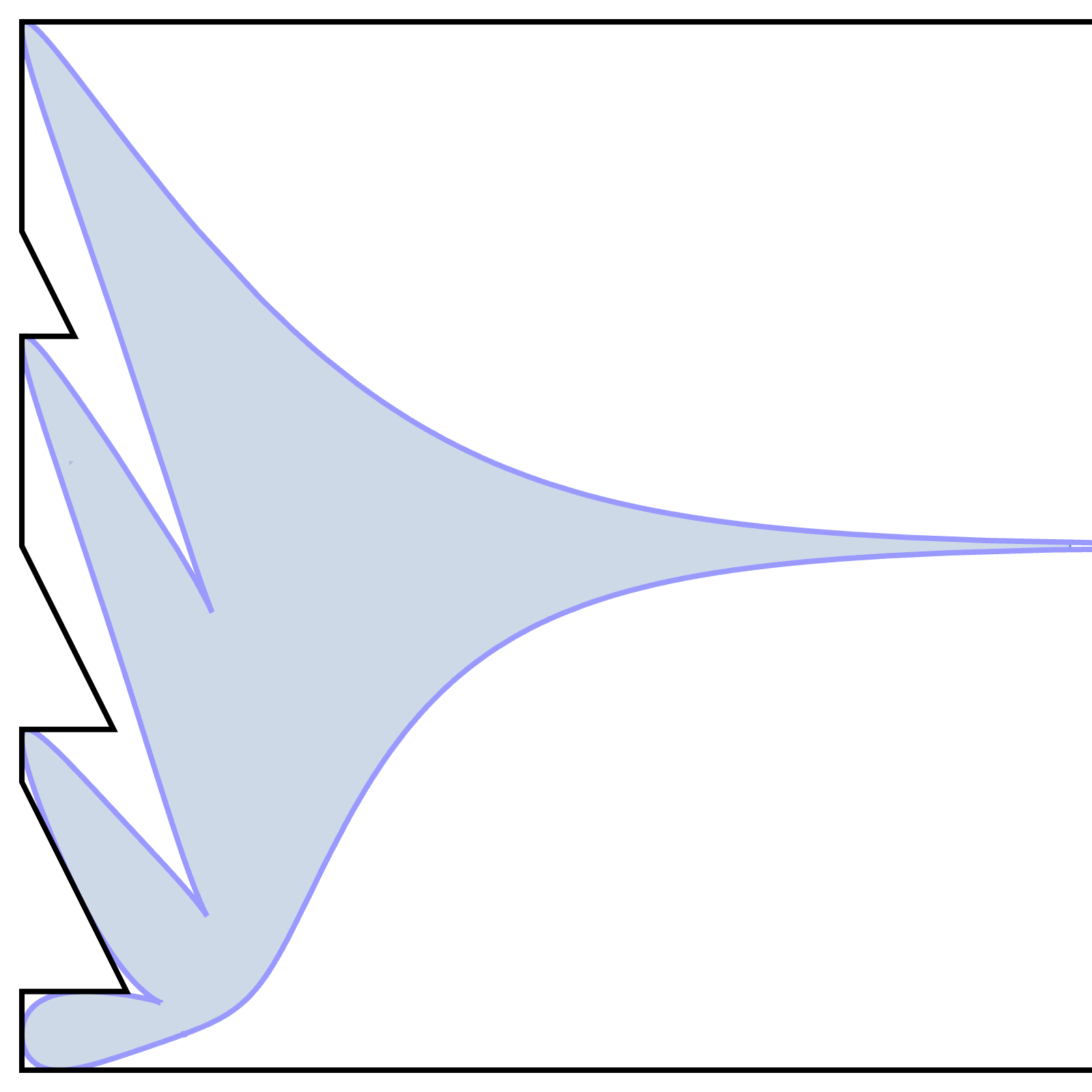}
	\caption{Examples of the 
	liquid region and the frozen boundary in the $(\tau,\rho)$ plane,
	with $\gamma=1/3,1$, and $3$, and the same initial conditions
	$\vec a=(0, 0.1, 0.2, 0.6, 1)$,
	$\vec C=(0.05, 0.45, 0.8, 1)$.
	The bounding polygonal region 
	indicates where the walks trajectories
	may lie, with the vertical straight
	pieces being the initial densely packed clusters
	of particles $\vec x$.
	Outside the liquid region, there are either no walks, or
	the walks are densely packed and move deterministically straight, horizontally
	or diagonally.}
	\label{fig:frozen_boundaries}
\end{figure}

For any $(\tau,\rho)\in \mathcal{D}$, let 
$\omega=\omega(\tau,\rho)$ be the unique root of the algebraic equation
\begin{equation*}
	\omega \ssp
	F\Bigl( e^{-\gamma \rho}\ssp\frac{1-\omega}{1-e^{\gamma\tau}\omega} \Bigr)
	=
	e^{-\gamma(\tau+1)}
\end{equation*}
in the upper half complex plane.
The existence and uniqueness of the complex root
of this equation (equivalent to \eqref{eq:critical_points_equation_exponentiated})
follow from
\Cref{sub:crit_pt} and the change 
of variables \eqref{eq:omega_for_beta_defn_in_last_step}.
With all this notation in place,
we can now formulate the main asymptotic result
of the paper:

\begin{theorem}
	\label{thm:bulk_limit}
	For any $(\tau,\rho)\in \mathcal{D}$, in the limit regime described above, we have
	\begin{equation*}
		\lim_{m\to+\infty}
		(-1)^{\Delta t}
		e^{\gamma(\tau+\rho)\Delta t}
		K_{\mathrm{walks}}
		\left( \lfloor \rho m \rfloor +\Delta p, \lfloor \tau m \rfloor +\Delta t;
		\lfloor \rho m \rfloor , \lfloor \tau m \rfloor  \right)
		=
		\mathbf{1}_{\Delta t=\Delta p=0}
		-
		\mathsf{B}_{\omega}(\Delta t,\Delta p)
	\end{equation*}
	for any fixed $\Delta t,\Delta p\in \mathbb{Z}$.
\end{theorem}

Let us make two remarks about \Cref{thm:bulk_limit}.
First, the factor 
$(-1)^{\Delta t}
e^{\gamma(\tau+\rho)\Delta t}$
in front of $K_{\mathrm{walks}}$
is a 
so-called ``gauge transformation'' of the correlation kernel
which does not change the 
determinants in \eqref{eq:det_def_proc_intro},
and thus preserves the determinantal process. Therefore,
\Cref{thm:bulk_limit} states that
the point process of the random walks converges locally
(at the lattice level, in a neighborhood of the global position 
$(\tau,\rho)$)
to the \emph{complement} of the point process coming from the
unique ergodic
translation invariant Gibbs measure on lozenge tilings of the whole plane
with parameter $\omega$. 
The complement arises by the
Kerov's complementation principle (see, for example,
\cite[Appendix A.3]{borodin2000b})
because our correlation kernel is $\mathbf{1}-\mathsf{B}_\omega$, where $\mathbf{1}$
is the identity operator.

\smallskip

Second, let us discuss the 
densely
packed clusters assumption
\eqref{eq:densely_packed_s2}.
On the one hand, it restricts the generality of
the initial conditions. On the other hand, it leads to
elegant formulas for the global frozen boundary, and simplifies the
technical part of the analysis.
The bulk limit asymptotics of \Cref{thm:bulk_limit}
should follow for general initial data $\vec x$ by a more
delicate steepest descent analysis of our kernel,
similarly to what is done in \cite{GorinPetrov2016universality} 
for the $q=1$ noncolliding random walks with general initial data.
We do not pursue this analysis here.
See also \cite{duse2015asymptotic}, \cite{Duse2015_partII}
for limit shape and fluctuation results on uniformly random lozenge tilings 
with more general boundary conditions.

\medskip
Finally, we make a conjecture about the final absorbing time of the 
noncolliding $q$-exchangeable random walks:

\begin{remark}[Asymptotics of the absorption time of $\Upsilon_m$]
	\label{rmk:absorption_time}
	Note that the
	liquid region is unbounded.
	More precisely, the
	frozen boundary has an asymptote approaching
	$\rho=1$ as $\tau\to+\infty$. 
	This implies that the \emph{absorption time} of the Markov chain $\Upsilon_m$,
	that is, the random time 
	\begin{equation*}
		t_{\mathrm{abs}}(m)\coloneqq
		\min\left\{ t\in \mathbb{Z}_{\ge0}\colon y_1(t)=m-1 \right\},
	\end{equation*}
	grows faster than $m$.
	Based on the result of Mutafchiev
	\cite{mutafchiev2006size}
	on unrestricted random plane partitions, we conjecture that 
	$t_{\mathrm{abs}}(m)\sim m\log m$ as $m\to+\infty$.
	
	It should be possible to obtain a more precise 
	behavior with the generating function coefficients technique as in
	\cite{mutafchiev2006size} (based on Hayman’s admissible functions \cite{hayman1956genf}).
	Indeed, this is because
	our ensemble of plane partitions coming from $\Upsilon_m$ has an explicit
	partition function 
	(\Cref{prop:q_vol_Upsilon_partition_function}).
\end{remark}

\section{From lozenge tilings to noncolliding $q$-exchangeable walks}
\label{sec:from_tilings_to_walks}

Here we recall the result from \cite{petrov2022noncolliding}
which shows how the noncolliding $q$-exchangeable random walks
$\Upsilon_m$ arise as a limit of the 
$q$-distributed random lozenge tilings.

\subsection{$q$-distributed lozenge tilings of sawtooth polygons}
\label{sub:tilings}

Let $\mathbb{GT}_N$ be the set of all partitions of length
$N$, that is, $N$-tuples of nonnegative integers $\lambda =
(\lambda_1 \geq \lambda_2 \geq \dots \lambda_N \ge0)$.
Denote $|\lambda|=\lambda_1+\ldots+\lambda_N $.
We
say that $\mu\in \mathbb{GT}_{N-1}$ interlaces with
$\lambda\in \mathbb{GT}_N$, denoted by $\mu \prec\lambda$,
if $\lambda_1\ge \mu_1\ge \lambda_2\ge \ldots\ge
\lambda_{N-1}\ge \mu_{N-1}\ge \lambda_N$.
For a sequence
\begin{equation}
\Lambda = (\varnothing \prec \Lambda^{(1)} \prec \Lambda^{(2)} \prec \dots \prec \Lambda^{(N-1)} \prec \Lambda^{(N)}),
\qquad \Lambda^{(j)}\in \mathbb{GT}_j,
\label{eq:Lambda_interlacing_sequence}
\end{equation}
define its volume by
\begin{equation*}
	\mathop{\mathrm{volume}} (\Lambda) = \sum_{m=1}^{N-1} |\Lambda^{(m)}|.
\end{equation*}

Fix $\lambda\in \mathbb{GT}_N$, and consider the 
probability measure on sequences
\eqref{eq:Lambda_interlacing_sequence}
with fixed top row $\Lambda^{(N)}=\lambda$, and 
probability weight proportional to $q^{-\mathrm{volume}(\Lambda)}$.
Denote this probability measure by 
$\mathscr{T}_N^{(\lambda)}$.
We may express the partition
function of 
$\mathscr{T}_N^{(\lambda)}$ as follows
(see, e.g., 
\cite[Section 3]{Petrov2012} for more detail):
\begin{equation}
	\label{eq:Z_q_partition_function}
	Z[\mathscr{T}_N^{(\lambda)}]=\sum_{\Lambda\colon \Lambda^{(N)}=\lambda}
	q^{-\mathrm{volume}(\Lambda)}=
	s_{\lambda} (q^{1-N}, q^{2-N}, \dots, q^{-1}, 1),
\end{equation}
where $s_\lambda$ is the Schur symmetric polynomial
$s_\lambda(z_1,\ldots,z_N )=\det[z_i^{\lambda_j+N-j}]_{i,j=1}^{N}\prod\limits_{1\le 
i<j\le N}(z_i-z_j)^{-1}$.
The right-hand side of \eqref{eq:Z_q_partition_function}
may also be simplified to the product form:
\begin{equation*}
	Z[\mathscr{T}_N^{(\lambda)}]=
	q^{|\lambda|(1-N)}
	\prod_{1\le i<j\le N}\frac{q^{\lambda_i-i}-q^{\lambda_j-j}}{q^{-i}-q^{-j}}.
\end{equation*}

The probability measure $\mathscr{T}_N^{(\lambda)}$ has a bijective
interpretation as a distribution on lozenge tilings of a 
sawtooth polygon of depth $N$ and fixed top boundary
determined by $\lambda$.
Define
\begin{equation}
	\label{eq:p_n_i_and_lambdas}
	p_i^n=\Lambda^{(n)}_i-i,\qquad 
	1\le i\le n\le N.
\end{equation}
Under $\mathscr{T}_N^{(\lambda)}$, the quantities $p_i^n$ form a random
integer
array 
$\mathfrak{P}=\{p_i^n \}_{1\le i\le n\le N}$
satisfying the interlacing constraints:
\begin{equation*}
	p_{i+1}^n < p_{i}^{n-1}\le p_{i}^n
\end{equation*}
(for all $i$'s and $n$'s for which these inequalities can be written out).
Viewing each
$p_i^n$ as the 
coordinate of the center of a vertical lozenge
\lozv{2.5pt}{.2}\ssp, we may complete the 
tiling in a unique way by the other two types of lozenges.
This leads to a corresponding tiling of a sawtooth polygon
as in 
\Cref{fig:lozenge_and_schemes}.

\begin{figure}[htpb]
	\centering
	\includegraphics[width=.6\textwidth]{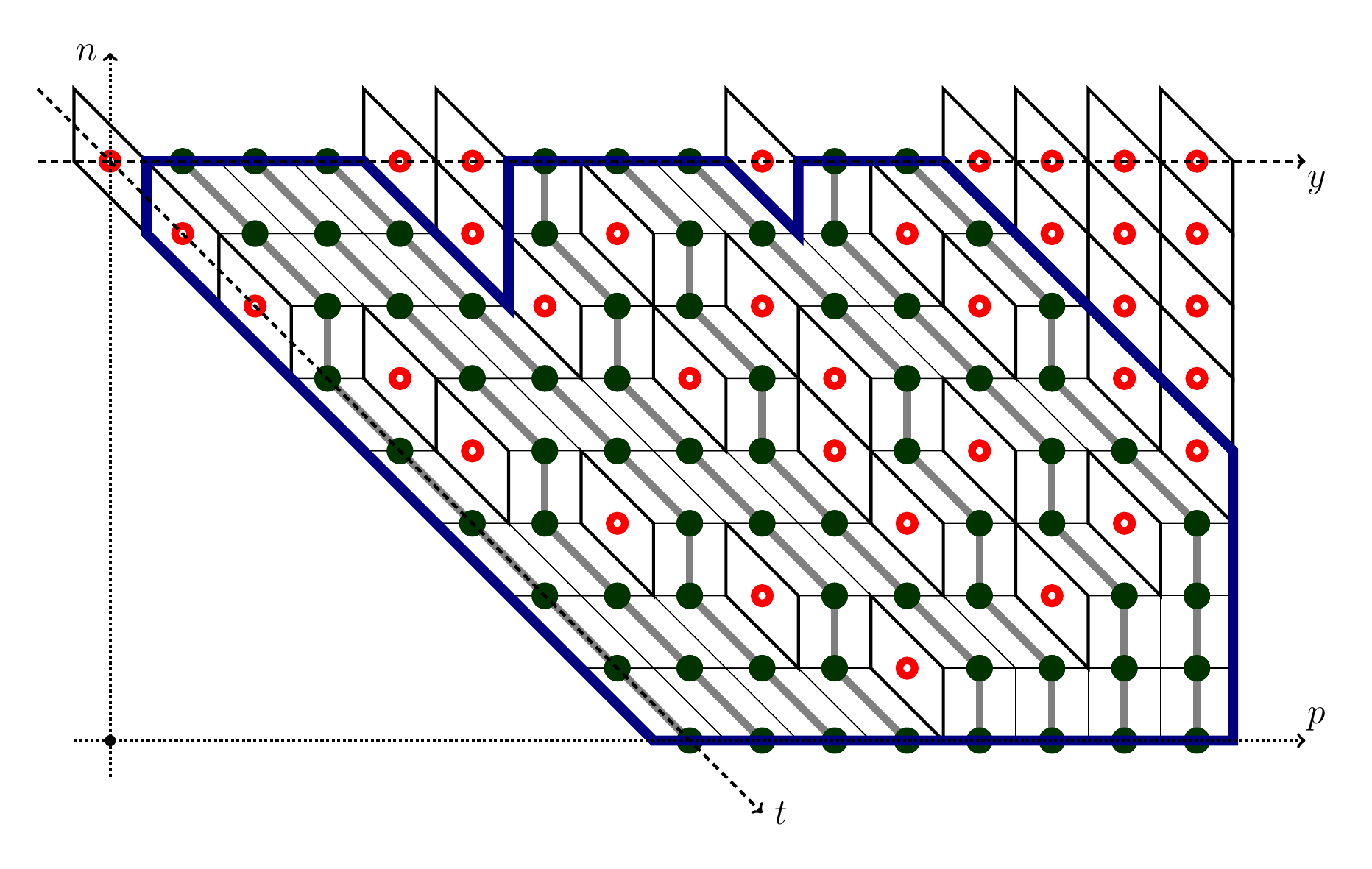}
	\caption{An example of a lozenge tiling of a sawtooth polygon of height $N=8$ with top row
		$\lambda=(16, 16, 16, 16, 14, 11, 11, 8)$. The particle array $\mathfrak{P}=\{p^n_i \}$ 
		consists of
		the red circle dots placed at the centers of the vertical lozenges, with coordinates relative to the $(p,n)$
		coordinate system (with dotted axes). }
	\label{fig:lozenge_and_schemes}
\end{figure}

\begin{remark}
	\label{rmk:T_N_vs_M_m}
	Note that the measure $\mathscr{T}_N^{(\lambda)}$ on lozenge
	tilings of a sawtooth polygon with weights proportional
	to $q^{-\mathrm{volume}}$ is different
	from the measure $\mathscr{M}_m^{(\vec x)}$ described
	in \Cref{sub:gibbs_q_lozenge_partition_function} above.
	In particular, the tilings under
	$\mathscr{M}_m^{(\vec x)}$ live in an infinite domain, 
	and are weighted proportionally to $q^{+\mathrm{volume}}$.
	In the next \Cref{sub:limit_to_walks}
	we explain how the measures
	$\mathscr{T}_N^{(\lambda)}$ as $N\to+\infty$ and special choice of $\lambda$
	lead to 
	$\mathscr{M}_m^{(\vec x)}$.
\end{remark}

\subsection{Limit transition to random walks}
\label{sub:limit_to_walks}

Now let us connect the probability measure
$\mathscr{T}_N^{(\lambda)}$ to the noncolliding
$q$-exchangeable random walks $\Upsilon_m$ from
\Cref{def:Upsilon_m}. Observe that for $\lambda\in \mathbb{GT}_N$,
we have:
\begin{equation*}
	\mathscr{T}_N^{(\lambda)}\bigl( \Lambda^{(N-1)}=\mu \bigr)
	=
	\frac{q^{-|\mu|}\ssp Z[\mathscr{T}_{N-1}^{(\mu)}]}{Z[\mathscr{T}_N^{(\lambda)}]}
	=
	q^{(N-1)\left( |\lambda|-|\mu| \right)}
	\frac{s_\mu(1,q,\ldots,q^{N-2} )}{s_\lambda(1,q,\ldots,q^{N-2},q^{N-1} )}.
\end{equation*}
Indeed, the first equality follows from the $q^{-\mathrm{volume}}$ probability
weights, and the second equality
is due to \eqref{eq:Z_q_partition_function}
and the homogeneity of the Schur polynomials.

Fix $m\ge1$ and $\vec x,\vec y\in \mathbb{W}_m$.
Let $\lambda\in \mathbb{GT}_N$ and $\mu\in \mathbb{GT}_{N-1}$
depend on $\vec x,\vec y$
as follows (here and below we assume that $N$ is sufficiently large):
\begin{equation}
	\label{eq:lambda_mu_via_x_y}
	\begin{split}
		\left\{\lambda_1-1, \lambda_2-2, \ldots, \lambda_N-N\right\} & =\{0,1,2, \ldots, N+m-1 \} \setminus \left\{x_1, \ldots, x_m \right\}, \\ 
		\left\{\mu_1-1, \mu_2-2, \ldots, \mu_{N-1}-(N-1)\right\} & =\{1,2, \ldots, N+m-1 \} \setminus \left\{y_1+1, \ldots, y_m+1 \right\}.
	\end{split}
\end{equation}
This choice of $\lambda$ and $\mu$
means that we pass from a lozenge
tiling to
nonintersecting paths 
avoiding 
the lozenges of type \lozv{2.5pt}{.2}\ssp, see
\Cref{fig:lozenge_and_schemes}.
Moreover, we choose the boundary conditions such that
the number $m$ of paths stays fixed as $N$ grows.

The following result is a particular case of
\cite[Section 3]{petrov2022noncolliding}
with $t=q$:
\begin{proposition}
	\label{prop:convergence_of_walks}
	With the above notation, for fixed $m$ and $\vec x,\vec y\in \mathbb{W}_m$,
	we have 
	\begin{equation*}
		\lim_{N \to +\infty} 
		\mathscr{T}_N^{(\lambda)}\bigl( \Lambda^{(N-1)}=\mu \bigr)
		=\Upsilon_m(\vec x,\vec y),
	\end{equation*}
	where $\Upsilon_m$ is given by \eqref{eq:intro_Upsilon}.
\end{proposition}

\Cref{prop:convergence_of_walks} states that the limiting distribution of the 
nonintersecting paths in \Cref{fig:lozenge_and_schemes}
is the same as the distribution of the trajectory of the noncolliding $q$-exchangeable random walks $\Upsilon_m$.
In \Cref{fig:lozenge_and_schemes},
the noncolliding paths live in the coordinate
system $(y,t)$ (with dashed axes), and converge $\Upsilon_m$
in an arbitrary finite neighborhood of the point $y=t=0$.
In \Cref{sec:limit_in_kernel} below
we use this limiting relation to write down the 
correlation kernel for the 
random walks $\Upsilon_m$.

\section{Limit transition in the kernel}
\label{sec:limit_in_kernel}

\subsection{Correlation kernel for $q$-distributed lozenge
tilings of sawtooth polygons}

We begin by recalling Theorem 4.1 from \cite{Petrov2012} about the correlation
kernel of the measure $\mathscr{T}_N^{(\lambda)}$ 
on lozenge
tilings of the sawtooth polygon with top row $\lambda$
described in \Cref{sub:tilings}
above.
By the results of 
\cite{Kenyon1997LocalStat}
(see also \cite[Section 7]{Borodin2009}),
this measure 
is a determinantal point process in the sense that
for any $\ell\ge 1$ and any pairwise distinct
$(p_1,n_1),\ldots,(p_\ell,n_\ell)\in \mathbb{Z}^2$ we have
\begin{equation}
	\label{eq:corr_f_def}
	\mathop{\mathrm{Prob}}\left( \textnormal{the 
	random array $\mathfrak{P}$ contains all
	$(p_1,n_1),\ldots,(p_\ell,n_\ell)$} \right)
	=
	\det\left[ K_{\mathrm{loz}}(p_i,n_i;p_j,n_j) \right]_{i,j=1}^{\ell}.
\end{equation}
The kernel $K_{\mathrm{loz}}$ is computed in \cite[Theorem 4.1]{Petrov2012}
and is given by the following 
double contour integral
formula:
\begin{equation}
	\label{eq:kernel_from_Pet14}
	\begin{split}
		&K_{\mathrm{loz}}(p_1,n_1;p_2,n_2)=
		-\mathbf{1}_{n_2<n_1}\mathbf{1}_{p_2\le p_1}q^{n_2(p_1-p_2)}
		\frac{(q^{p_1-p_2+1};q)_{n_1-n_2-1}}{(q;q)_{n_1-n_2-1}}
		\\&\hspace{10pt}
		+\frac{(q^{N-1};q^{-1})_{N-n_1}}{(2\pi\mathbf{i})^2}
		\oiint
		\frac{dz\ssp dw}{w}
		\frac{q^{n_2(p_1-p_2)}z^{n_2}}{w-z}
		\\&\hspace{20pt}\times
		{}_2\phi_1(q^{-1},q^{n_1-1};q^{N-1}\mid q^{-1};w^{-1})\ssp
		\frac{(zq^{1-p_2+p_1};q)_{N-n_2-1}}{(q;q)_{N-n_2-1}}
		\prod_{r=1}^{N}\frac{w-q^{\lambda_r-r-p_1}}{z-q^{\lambda_r-r-p_1}}.
	\end{split}
\end{equation}
Here the points
$(p_1,n_1),(p_2,n_2)\in \mathbb{Z}^2$ are such that $1\le n_1\le N$,
$1\le n_2\le N-1$. The $z$ and $w$ integration contours are
counterclockwise and do not intersect. 
The $z$ contour encircles 
$q^{p_2-p_1},q^{p_2-p_1+1},\ldots,q^{\lambda_1-1-p_1}$
and not 
$q^{p_2-p_1-1},q^{p_2-p_1-2},\ldots$.
The $w$ contour is sufficiently large and goes around $0$
and the $z$ contour.
Finally, $_2\phi_1$ in \eqref{eq:kernel_from_Pet14}
is the (in this case, terminating) Gauss $q$-hypergeometric
function given by 
\begin{equation}
	\label{eq:2phi1_defn_our_case}
	{}_2\phi_1(q^{-1},q^{n_1-1};q^{N-1}\mid q^{-1};w^{-1})
	=
	\sum_{j=0}^{n_1-1}
	\frac{(q^{-1};q^{-1})_j(q^{n_1-1};q^{-1})_j}{(q^{N-1};q^{-1})_j}
	\frac{w^{-j}}{(q^{-1};q^{-1})_j}.
\end{equation}

In the rest of this section we consider the $N\to+\infty$
limit of the 
kernel $K_{\mathrm{loz}}$ in the regime leading to the 
noncolliding $q$-exchangeable random walks (as discussed in \Cref{sub:limit_to_walks}),
and prove \Cref{thm:det_kernel_for_walks_intro}.

\subsection{Rewriting the kernel}

Fix $t_1\ge0$, $t_2>0$, and let $n_1=N-t_1,n_2=N-t_2$ (throughout the rest of the section
we assume that $N$ is sufficiently large). 
Change the integration variables in \eqref{eq:kernel_from_Pet14} 
as $\tilde z=zq^{p_1}$,
$\tilde w=wq^{p_1}$, and then rename back to $z,w$. We have
\begin{equation}
	\begin{split}
		\label{eq:K_rewrite_1}
		&q^{N(p_2-p_1)}K_{\mathrm{loz}}(p_1,N-t_1;p_2,N-t_2)=
		-\mathbf{1}_{t_2>t_1}\mathbf{1}_{p_2\le p_1}q^{(-t_2)(p_1-p_2)}
		\frac{(q^{p_1-p_2+1};q)_{t_2-t_1-1}}{(q;q)_{t_2-t_1-1}}
		\\&\hspace{10pt}
		+\frac{(q^{N-1};q^{-1})_{t_1}}{(2\pi\mathbf{i})^2}
		\oiint
		\frac{dz \ssp dw}{w}
		\frac{q^{(-t_2)(p_1-p_2)}z^{N-t_2}q^{-p_1(N-t_2)}}{w-z}
		\\&\hspace{20pt}\times
		{}_2\phi_1(q^{-1},q^{N-t_1-1};q^{N-1}\mid q^{-1};w^{-1}q^{p_1})
		\frac{(zq^{1-p_2};q)_{t_2-1}}{(q;q)_{t_2-1}}
		\prod_{r=1}^{N}\frac{w-q^{\lambda_r-r}}{z-q^{\lambda_r-r}}.
	\end{split}
\end{equation}
Here the $z$ contour encircles only $q^{p_2},q^{p_2+1},q^{p_2+2},\ldots $ and no other $z$ poles of the integrand,
and the $w$ contour goes around $0$ and the $z$ contour.

The factor $q^{N(p_2-p_1)}$ is a so-called ``gauge transformation'' of the correlation kernel
which does not change the 
determinants in \eqref{eq:corr_f_def} and thus preserves the determinantal process.
In general, by a gauge transformation we mean replacing a kernel
$K(p_1,t_1;p_2,t_2)$ by 
$\frac{f(p_1,t_1)}{f(p_2,t_2)}K(p_1,t_1;p_2,t_2)$, where $f$ is nowhere vanishing.

\medskip

In the next step, we use the fact that the top row
$\lambda$ depends on $N$ as in \eqref{eq:lambda_mu_via_x_y}.
Here $\vec x\in \mathbb{W}_m$ is the fixed initial configuration of the noncolliding $q$-exchangeable random walks $\Upsilon_m$.
Relation \eqref{eq:lambda_mu_via_x_y} allows to express the last product 
over $r$
in the integrand in \eqref{eq:K_rewrite_1}
in terms of the $x_j$'s. After necessary simplifications, we
obtain
\begin{equation}
	\label{eq:K_rewrite_2}
	\begin{split}
		&q^{N(p_2-p_1)}K_{\mathrm{loz}}(p_1,N-t_1;p_2,N-t_2)=
		-\mathbf{1}_{t_2>t_1}\mathbf{1}_{p_2\le p_1}q^{(-t_2)(p_1-p_2)}
		\frac{(q^{p_1-p_2+1};q)_{t_2-t_1-1}}{(q;q)_{t_2-t_1-1}}
		\\&\hspace{10pt}
		+\frac{(q^{N-1};q^{-1})_{t_1}}{(2\pi\mathbf{i})^2}
		\oiint
		\frac{dz \ssp dw}{w}
		\frac{z^{-t_2}q^{p_2t_2}}{w-z}
		q^{-Np_1}w^N{}_2\phi_1(q^{-1},q^{N-t_1-1};q^{N-1}\mid q^{-1};w^{-1}q^{p_1})
		\\&\hspace{20pt}
		\times
		\frac{(zq^{1-p_2};q)_{t_2-1}}{(q;q)_{t_2-1}}
		\prod_{j=0}^{N+m-1}\frac{1-q^{j}/w}{1-q^{j}/z}
		\prod_{r=1}^{m}\frac{1-q^{x_r}/z}{1-q^{x_r}/w}.
		\end{split}
\end{equation}
The integration contours in \eqref{eq:K_rewrite_2}
are the same as in \eqref{eq:K_rewrite_1}.

\subsection{Exchanging the contours}
\label{sub:exchange_contours_kernel}

We now move the $w$ contour inside the $z$ contour in \eqref{eq:K_rewrite_2}.
The integral stays the same, 
but we need to add a contour
integral of minus the residue $\mathop{\mathrm{Res}}_{z=w}$ over the $w$ contour around $0$.
The residue is equal to
\begin{equation}
	\label{eq:K_rewrite_residue}
	\begin{split}
		&-
		\mathop{\mathrm{Res}}\nolimits_{z=w}=
		\frac{(q^{N-1};q^{-1})_{t_1}}{2\pi\mathbf{i}}
		\frac{1}{w}
		w^{-t_2}q^{p_2t_2}
		\\&\hspace{80pt}\times
		q^{-Np_1}w^N{}_2\phi_1(q^{-1},q^{N-t_1-1};q^{N-1}\mid q^{-1};w^{-1}q^{p_1})
		\frac{(wq^{1-p_2};q)_{t_2-1}}{(q;q)_{t_2-1}}.
	\end{split}
\end{equation}

\begin{lemma}
	\label{lemma:K_rewrite_residue}
	The integral in $w$ of 
	\eqref{eq:K_rewrite_residue}
	over a small contour around zero is equal to
	\begin{equation*}
		\mathbf{1}_{t_2>t_1}q^{(-t_2)(p_1-p_2)}
		\frac{(q^{p_1-p_2+1};q)_{t_2-t_1-1}}{(q;q)_{t_2-t_1-1}}.
	\end{equation*}
\end{lemma}
\begin{proof}
	Notice that ${}_2\phi_1(q^{-1},q^{N-t_1-1};q^{N-1}\mid q^{-1};w^{-1}q^{p_1})$ and $(wq^{1-p_2}; q)_{t_2-1}$ are Laurent polynomials
	in $w$.
	Therefore, the integral
	of 
	\eqref{eq:K_rewrite_residue}
	over a small contour around zero
	is simply the operation of picking the coefficient
	by $1/w$.

	By the $q$-binomial theorem, we can write
	\begin{equation}
		\label{eq:qpoch_series}
		\begin{split}
			(w q^{1-p_2}; q)_{t_2-1} &= (w^{-1} q^{p_2-t_2+1}; q)_{t_2-1} (-w q^{1-p_2})^{t_2-1} q^{\binom {t_2-1} 2} \\
			&= (-w q^{1-p_2})^{t_2-1} q^{\binom
			{t_2-1} 2}  \sum_{j=0}^{t_2 -1} (-1)^j
			q^{(p_2-t_2+1)j} q^{\binom{j}2}
			\frac{(q;q)_{t_2 -1}}{(q;q)_j
			(q;q)_{t_2-1-j}}   w^{-j}.
		\end{split}
	\end{equation}
	Using formula \eqref{eq:2phi1_defn_our_case}
	for ${}_2\phi_1$ and
	\eqref{eq:qpoch_series}, the product of the two Laurent polynomials has the form
	\begin{equation*}
		\sum_{j=0}^{N-t_1-1}  q^{p_1 j}
		\frac{(q^{N-t_1-1};q^{-1})_j}{(q^{N-1};q^{-1})_j}
		w^{-j} \times \sum_{j=0}^{t_2 -1} (-1)^j
		q^{(p_2-t_2+1)j} q^{\binom{j}2} \frac{(q;q)_{t_2-1}}
		{(q;q)_j (q;q)_{t_2-1-j}}   w^{-j}=
		\sum_{j=0}^{N+t_2-t_1-2} \mathsf{d}_j w^{-j}.
	\end{equation*}
	
	From the prefactor in \eqref{eq:K_rewrite_residue} we see that we need to compute
	the sum $\mathsf{d}_{N-1}$ which has the form
	\begin{equation*}
		\mathsf{d}_{N-1} = \sum_{m=0}^{N-1} q^{p_1(N-1-m)}q^{\binom m 2} (-1)^m q^{(p_2 - t_2 + 1) m} \frac{(q^{N-t_1-1};q^{-1})_{N-1-m}}
		{(q^{N-1};q^{-1})_{N-1-m}} \frac{(q;q)_{t_2-1}}{(q;q)_m (q;q)_{t_2-1-m}}.
	\end{equation*}
	Observe that $m$-th term in the sum vanishes for $m>t_2$ or $m<t_1$, so 
	the limits of the summation are 
	$\mathbf{1}_{t_2 > t_1} \sum_{m=t_1} ^{t_2-1}$.
	Rearranging the terms and relabeling $k=m-t_1$, we have
	\begin{equation*}
		\begin{split}
			\mathsf{d}_{N-1} &= \mathbf{1}_{t_2 > t_1} (-1)^{t_1} \frac{(q^{t_2 - t_1},q)_{t_1}}{(q^{N - t_1},q)_{t_1}} q^{p_1(N-1-t_1)+\binom{t_1}2 + (p_2 - t_2 + 1) t_1} \\
			&\hspace{90pt}\times
			\sum_{k=0} ^{t_2-t_1-1} q^{  (t_1 - p_1 + p_2 - t_2 +1) k} q^{\binom k 2} (-1)^{k} 
			\frac{(q;q)_{t_2-t_1-1}}{(q;q)_{k}(q;q)_{t_2-t_1-1-k}}.
		\end{split}
	\end{equation*}
	Applying the $q$-binomial theorem, we can simplify this sum to:
	\begin{equation*}
		\mathsf{d}_{N-1} = \mathbf{1}_{t_2 > t_1} (-1)^{t_1}
		\frac{(q^{t_2 - t_1} ;q)_{t_1} }{(q^{N - t_1},q)_{t_1}}
		q^{p_1(N-1-t_1)+\binom{t}2 + (p_2 - t_2 + 1) t_1}
		(q^{t_1 - p_1 + p_2 - t_2 +1}; q)_{t_2-t_1-1}.
	\end{equation*}
	Putting together all factors from the above computation completes the proof. 
\end{proof}

By \Cref{lemma:K_rewrite_residue}, the kernel takes the form
\begin{equation}
	\label{eq:K_rewrite_3}
	\begin{split}
		&q^{N(p_2-p_1)}K_{\mathrm{loz}}(p_1,N-t_1;p_2,N-t_2)=
		\mathbf{1}_{t_2>t_1}\mathbf{1}_{p_2> p_1}q^{(-t_2)(p_1-p_2)}
		\frac{(q^{p_1-p_2+1};q)_{t_2-t_1-1}}{(q;q)_{t_2-t_1-1}}
		\\&\hspace{10pt}
		+\frac{(q^{N-1};q^{-1})_{t_1}}{(2\pi\mathbf{i})^2}
		\oiint
		\frac{dz \ssp dw}{w}
		\frac{z^{-t_2}q^{p_2t_2}}{w-z}
		\ssp q^{-Np_1}w^N{}_2\phi_1(q^{-1},q^{N-t_1-1};q^{N-1}\mid q^{-1};w^{-1}q^{p_1})
		\\&\hspace{20pt}\times
		\frac{(zq^{1-p_2};q)_{t_2-1}}{(q;q)_{t_2-1}}
		\prod_{j=0}^{N+m-1}\frac{1-q^{j}/w}{1-q^{j}/z}
		\prod_{r=1}^{m}\frac{1-q^{x_r}/z}{1-q^{x_r}/w}.
	\end{split}
\end{equation}
The integration contours in \eqref{eq:K_rewrite_3} have changed, namely, the $w$ contour is an arbitrarily small 
circle around $0$, and the $z$ contour goes around $q^{p_2},q^{p_2+1},q^{p_2+2},\ldots $, $0$, the $w$ contour, and 
encircles no other $z$ poles of the integrand.
In \eqref{eq:K_rewrite_3}, the first summand is a combination of the 
residue from \Cref{lemma:K_rewrite_residue} and the first summand in the previous expression 
\eqref{eq:K_rewrite_2}.

\subsection{Limit of the $q$-hypergeometric function}
\label{sub:hyp_limit_for_kernel}

After exchanging the $z$ and $w$ contours, 
$|w|$ can be taken arbitrarily small. This allows to take a limit in $N$ in the
part of the integrand in \eqref{eq:K_rewrite_3} 
containing the $q$-hypergeometric function $_2\phi_1$ (observe that this is 
essentially the only dependence on $N$ left in the integrand). 
Denote
\begin{equation*}
	Q_N(w)\coloneqq 
	(q^{N-1};q^{-1})_{t_1}
	q^{-Np_1}w^N{}_2\phi_1(q^{-1},q^{N-t_1-1};q^{N-1}\mid q^{-1};w^{-1}q^{p_1}).
\end{equation*}
Then
\begin{equation*}
	\begin{split}
		Q_N(w)
		&
		=
		(q^{N-1};q^{-1})_{t_1}
		q^{-Np_1}w^N
		\sum_{j=0}^{N-t_1-1}
		\frac{(q^{N-t_1-1};q^{-1})_j}{(q^{N-1};q^{-1})_j}
		\ssp w^{-j}q^{jp_1}
		\\&
		=
		(q^{N-1};q^{-1})_{t_1}
		\sum_{k=0}^{N-t_1-1}
		\frac{(q^{N-t_1-1};q^{-1})_{N-t_1-1-k}}{(q^{N-1};q^{-1})_{N-t_1-1-k}}
		(w/q^{p_1})^{t_1+1+k},
	\end{split}
\end{equation*}
where we used \eqref{eq:2phi1_defn_our_case} and in the last line flipped the summation index as $k=N-t_1-1-j$.
We have 
\begin{equation*}
	\lim_{N\to+\infty}(q^{N-1};q^{-1})_{t_1}=1.
\end{equation*}
Next, in each $k$-th term in the sum we have (for $k$ fixed):
\begin{multline*}
	\frac{(q^{N-t_1-1};q^{-1})_{N-t_1-1-k}}{(q^{N-1};q^{-1})_{N-t_1-1-k}}
	\ssp
	(w/q^{p_1})^{t_1+1+k}
	\\=
	(w/q^{p_1})^{t_1+1+k}\prod_{i=0}^{N-t_1-k-2}
	\frac{1-q^{k+i+1}}{1-q^{t_1+k+i+1}}
	\to
	(w/q^{p_1})^{t_1+1+k}\prod_{i=0}^{t_1-1}
	(1-q^{k+i+1}),
	\qquad N\to+\infty,
\end{multline*}
and because $|w|$ is small, the convergence is uniform in $k$ and $w$.
Thus, we have 
\begin{equation}
	\label{eq:limit_of_hyp_in_kernel}
	\lim_{N\to+\infty}Q_N(w)=
	\sum_{k=0}^{\infty}
	(w/q^{p_1})^{t_1+1+k}
	(q^{k+1};q)_{t_1},
\end{equation}
uniformly in $w$ for small $|w|$.

\begin{lemma}
	\label{lemma:hyp_limit_second_lemma}
	The sum in the right-hand side of
	\eqref{eq:limit_of_hyp_in_kernel} is equal to 
	\begin{equation*}
		(w/q^{p_1})^{t_1+1}\frac{(q;q)_{t_1}}{(wq^{-p_1};q)_{t_1+1}}.
	\end{equation*}
\end{lemma}
\begin{proof}
	We have
	\begin{equation*}
		\begin{split}
		\sum_{k=0}^\infty (w q^{-p_1})^{t_1+1+k} (q^{k+1},q)_{t_1} 
		&= ( w q^{-p_1})^{t_1 + 1} (q;q)_{t_1} \sum_{k=0}^\infty (w q^{-p_1})^{k} \frac{(q,q)_{k+t_1}}{(q;q)_k (q;q)_{t_1}}\\
		&= ( w q^{-p_1})^{t_1 + 1}  \frac{(q;q)_{t_1}}{(w q^{-p_1}; q)_{t_1+1}},
		\end{split}
	\end{equation*}
	where we used the $q$-binomial theorem, and the series converges because
	$|w|$ is small.
\end{proof}

Putting together the formula \eqref{eq:K_rewrite_3} for the kernel $K_{\mathrm{loz}}$ and the 
results of the current \Cref{sub:hyp_limit_for_kernel},
we arrive at a $N\to+\infty$ limit of the kernel $K_{\mathrm{loz}}$. 
Denote
\begin{equation}
	\label{eq:K_loz_lim_defn}
	\begin{split}
		K_{\mathrm{loz}}^{\mathrm{lim}}(p_1,t_1;p_2,t_2)
		&\coloneqq
		\mathbf{1}_{t_2>t_1}\mathbf{1}_{p_2>p_1}
		\frac{q^{-t_2(p_1-p_2)} (q^{p_1-p_2+1};q)_{t_2-t_1-1}}{(q;q)_{t_2-t_1-1}}
		+\frac{q^{p_2t_2-p_1t_1-p_1}}{(2\pi\mathbf{i})^2}
		\oiint
		dz\ssp dw\ssp
		\frac{z^{-t_2}w^{t_1} }{w-z}
		\\&\hspace{65pt}\times
		\frac{(q;q)_{t_1}}{(wq^{-p_1};q)_{t_1+1}}
		\frac{(zq^{1-p_2};q)_{t_2-1}}{(q;q)_{t_2-1}}
		\frac{(w^{-1};q)_{\infty}}{(z^{-1};q)_{\infty}}
		\prod_{r=1}^{m}\frac{1-q^{x_r}/z}{1-q^{x_r}/w},
	\end{split}
\end{equation}
where 
the $w$ contour is an arbitrarily small 
circle around $0$, and the $z$ contour goes around $q^{p_2},q^{p_2+1},q^{p_2+2},\ldots $, $0$, the $w$ contour, and 
encircles
no other $z$ poles of the integrand.

The next proposition follows directly from the previous computations.
\begin{proposition}
	\label{prop:limit_of_K_loz_initial}
	For any fixed $t_1\ge0$, $t_2>0$, $p_1,p_2\in \mathbb{Z}$, we have
	\begin{equation*}
		\lim_{N\to+\infty}
		q^{N(p_2-p_1)} K_{\mathrm{loz}}(p_1,N-t_1;p_2,N-t_2)=
		K_{\mathrm{loz}}^{\mathrm{lim}}(p_1,t_1;p_2,t_2).
	\end{equation*}
\end{proposition}

\subsection{Particle-hole involution and time shift}

We are now in a position to 
derive 
\Cref{thm:det_kernel_for_walks_intro} 
from the limit transition of \Cref{prop:limit_of_K_loz_initial}.
Define
\begin{equation}
	\label{eq:K_walks}
	K_{\mathrm{walks}}(y_1,t_1;y_2,t_2)
	\coloneqq
	\mathbf{1}_{t_1=t_2}\mathbf{1}_{y_1=y_2}-
	q^{t_1(t_1+y_1)-t_2(t_2+y_2)}
	K_{\mathrm{loz}}^{\mathrm{lim}}(y_1+t_1,t_1;y_2+t_2,t_2).
\end{equation}
Observe that we performed two transformations to get $K_{\mathrm{walks}}$
from $K_{\mathrm{loz}}^{\mathrm{lim}}$ in \eqref{eq:K_walks}:
\begin{enumerate}[$\bullet$]
	\item First, the point process defined by the noncolliding walks
		(formed by the solid dots in \Cref{fig:lozenge_and_schemes})
		is the complement of the process defined by the particles $p^n_i$. 
		Therefore, by the 
		Kerov's complementation principle (see, for example,
		\cite[Appendix A.3]{borodin2000b}),
		the kernel for the walks is the 
		identity minus the kernel for the lozenges.
	\item Second, 
		the shifting of the variables $p_i=y_i+t_i$, $i=1,2$,
		corresponds to the passage from the 
		coordinate system $(p,n)$ (where $n=N-t$)
		to the coordinate system $(y,t)$, 
		see \Cref{fig:lozenge_and_schemes}.
\end{enumerate}
Finally, the factor in front of $K_{\mathrm{loz}}^{\mathrm{lim}}$
in \eqref{eq:K_walks} is simply a gauge transformation which
does not change the determinantal process.
One can readily verify that the resulting kernel $K_{\mathrm{walks}}$
\eqref{eq:K_walks}
is the same as
\eqref{eq:K_walks_intro}. This completes
the proof of \Cref{thm:det_kernel_for_walks_intro}.

\section{Asymptotic analysis}
\label{sec:asymptotic_analysis}

In this section, we perform the bulk asymptotic analysis of the
correlation kernel $K_{\mathrm{walks}}$ 
\eqref{eq:K_walks_intro}
of the process
$\Upsilon_m$ in the regime as
$q\to1$, $m\to \infty$, and the initial configuration $\vec
x$ forms a finite number of densely packed clusters.  We
make the latter assumption for technical convenience,
see, e.g., Duse--Metcalfe \cite{duse2015asymptotic}, \cite{Duse2015_partII}
for asymptotic results on uniformly random lozenge tilings 
with more general boundaries.
Using the steepest descent method, we prove \Cref{thm:bulk_limit}, that is, obtain the limit shape of the 
trajectories of $\Upsilon_m$, as well as the universal local
fluctuations of the paths which are governed by the incomplete beta kernel
introduced by Okounkov--Reshetikhin
\cite{okounkov2003correlation}.
The latter is a two-dimensional extension of the 
discrete sine kernel
introduced by Borodin--Okounkov--Olshanski \cite{borodin2000b}.

\subsection{Limit regime}
\label{sub:setup_asymptotics}

The limit regime we consider for the kernel
$K_{\mathrm{walks}}(p_1,t_1;p_2,t_2)$ 
\eqref{eq:K_walks_intro}
is as follows:
\begin{equation}
	\label{eq:asymptotic_regime}
	m\to+\infty; 
	\qquad 
	q=e^{-\gamma/m}\nearrow 1;
	\qquad 
	t_2=\lfloor m\tau \rfloor ,\quad t_1=t_2+\Delta t;
	\qquad 
	p_2=\lfloor m\rho \rfloor ,\quad p_1=p_2+\Delta p.
\end{equation}
Here $(\tau,\rho)\in \mathbb{R}^{2}_{\ge0}$, $\gamma>0$,
and the quantities $\Delta t=t_1-t_2,\Delta p=p_1-p_2\in \mathbb{Z}$
are fixed.
The regime with fixed differences $\Delta t,\Delta p$ is called \emph{bulk limit},
and it describes local correlations around the 
global point of observation $(\tau,\rho)\in \mathbb{R}^{2}_{\ge0}$.
Finally, we assume that the initial configuration
$\vec x\in \mathbb{W}_m$ scales as follows:
\begin{equation}
	\label{eq:scaling_initial_conf}
	x_i=\lfloor  m\ssp g\left( i/m  \right) \rfloor , 
	\quad 1\le i\le m;\qquad 
	g(u)=\sum_{i=1}^{L}(u+C_i)\mathbf{1}_{u \in [a_i,a_{i+1})},
\end{equation}
where $L\ge1$ is fixed (this is the number of clusters of densely
packed particles in $\vec x$), and
\begin{equation}
	\label{eq:a_C_relation}
	0<C_1<C_2<\ldots<C_L ,\qquad 
	0=a_1<a_2<\ldots<a_{L+1}=1 
\end{equation}
are the parameters of the clusters, and $g(u)$ is weakly
increasing with derivative $0$ or $1$.

We apply the standard steepest descent approach
outlined in \cite[Section 3]{Okounkov2002}. For this, 
we first rewrite the integrand in the double integral in
$K_{\mathrm{walks}}(p_1,t_1;p_2,t_2)$
\eqref{eq:K_walks_intro} as 
\begin{equation*}
	\begin{split}
		&
		-\frac{q^{-t_1-p_1}}{(2\pi\mathbf{i})^2}
		\frac{z^{-t_2}w^{t_1} }{w-z}
		\frac{(q;q)_{t_1}}{(wq^{-p_1-t_1};q)_{t_1+1}}
		\frac{(zq^{1-p_2-t_2};q)_{t_2-1}}{(q;q)_{t_2-1}}
		\frac{(w^{-1};q)_{\infty}}{(z^{-1};q)_{\infty}}
		\prod_{r=1}^{m}\frac{1-q^{x_r}/z}{1-q^{x_r}/w}
		\\
		&\hspace{2pt}
		=
		-
		\frac{q^{-t_1-p_1}}{(2\pi\mathbf{i})^2}
		\frac{(q;q)_{t_1}}{(q;q)_{t_2-1}}
		\frac{1}{(w-z)(1-zq^{-p_2})(1-zq^{-p_2-t_2})}\ssp
		\exp\left\{m \left(  
		S_m(w;t_1,p_1)-S_m(z;t_2,p_2) 
		\right) 
		\right\},
	\end{split}
\end{equation*}
where 
\begin{equation}
	\label{eq:def_S_m_new}
	\begin{split}
		S_m(w;t,p)&\coloneqq\frac{t}{m}\log w
		-\frac{1}{m}\sum_{i=0}^{t}\log (1-w q^{-p-t+i})
		\\&\hspace{90pt}
		+
		\frac{1}{m}\sum_{i=0}^{\infty}\log(1-w^{-1}q^{i})
		-\frac{1}{m}\sum_{r=1}^{m}\log(1-w^{-1}q^{x_r}).
	\end{split}
\end{equation}
Here we can take any branches of the logarithms 
so that $S_m$ is holomorphic in $w$ belonging to the upper half
complex plane. Indeed, any branches
taken the same in $S_m(w;t_1,p_1)$ and $S_m(z;t_2,p_2)$
produce the same
signs in the exponent in the integrand.

Using \eqref{eq:asymptotic_regime}--\eqref{eq:scaling_initial_conf},
let us define the 
limiting version of the function
$S_m$:
\begin{equation}
	\label{eq:def_S_new}
	\begin{split}
		S(w;\tau,\rho)&\coloneqq
		\tau \log w
		-
		\int_{0}^{\tau}
		\log\left( 1-w e^{\gamma (\tau+\rho-u)} \right)du
		\\&\hspace{90pt}
		+\int_0^\infty\log\left( 1-w^{-1}e^{-\gamma u} \right)du
		-\int_0^1 \log\left( 1-w^{-1}e^{-\gamma g(u)} \right)du.
	\end{split}
\end{equation}

\begin{lemma}
	\label{lemma:S_convergence}
	We have $S_m(w;\lfloor m \tau  \rfloor ,\lfloor m \rho  \rfloor )=S(w;\tau,\rho)+O(m^{-1})$ as $m\to +\infty$,
	uniformly for $w$
	and $(\tau,\rho)$
	belonging to compact subsets of $\left\{ w\colon \Im w>0 \right\}$
	and $\mathbb{R}_{\ge0}^2$, respectively.
\end{lemma}
\begin{proof}
	This follows from the convergence of the Riemann
	sums in \eqref{eq:def_S_m_new} 
	to the corresponding integrals in \eqref{eq:def_S_new},
	as the integrands are piecewise $C^1$ functions in $u$
	with norms uniformly bounded in $w,\tau,\rho$ belonging to compact subsets of their respective domains.
\end{proof}

Integrals in \eqref{eq:def_S_new}
can be expressed through the dilogarithm function
which has the series and the integral representations
\begin{equation}
	\label{eq:Li2}
	\mathrm{Li}_2(\xi)=\sum_{k=1}^{\infty}\frac{\xi^k}{k^2}=
	-\eta \int_0^\infty \log\left( 1-\xi e^{-\eta u} \right)du.
\end{equation}
The series converges for $|\xi|<1$, and 
the integral 
representation (valid for any $\eta>0$, but we will mostly use it with $\eta=\gamma$) follows a certain branch of the logarithm.
For example, we may choose a branch of 
$\mathrm{Li}_2(\xi)$ to have cut at $\xi\in \mathbb{R}_{\ge1}$.
We have
\begin{equation}
	\label{eq:Li2_derivatives}
	\mathrm{Li}_2'(\xi)=-\frac{\log(1-\xi)}{\xi},
	\qquad 
	\frac{\partial}{\partial \xi}\ssp\xi\ssp
	\frac{\partial}{\partial \xi}\ssp
	\mathrm{Li}_2(\xi)=\frac{1}{1-\xi}.
\end{equation}
With this notation and using \eqref{eq:scaling_initial_conf}, we have
\begin{equation}
	\label{eq:S_via_Li2}
	\begin{split}
		S(w;\tau,\rho)&=
		\tau\log w
		-
		\gamma^{-1}
		\mathrm{Li}_2( we^{\gamma\rho} )
		+
		\gamma^{-1}
		\mathrm{Li}_2\bigl( we^{\gamma(\rho+\tau)} \bigr)
		\\&\hspace{30pt}-
		\gamma^{-1}
		\mathrm{Li}_2( w^{-1} )
		+
		\gamma^{-1}
		\sum_{i=1}^L
		\left[ 
			\mathrm{Li}_2\bigl(w^{-1} e^{-\gamma(a_i+C_i)}\bigr)-
			\mathrm{Li}_2\bigl(w^{-1} e^{-\gamma(a_{i+1}+C_i)}\bigr)
		\right].
	\end{split}
\end{equation}

\subsection{Critical points and the frozen boundary}
\label{sub:crit_pt}

Let us count the critical points of 
$S(w;\tau,\rho)$
in the complex upper half-plane.
Recall that $w$ is a critical point if, by definition,
$S'(w;\tau,\rho)=0$, where the derivative is taken in $w$.
By looking at
$e^{\gamma w S'(w;\tau,\rho)}$, we see that the critical points
must satisfy the following algebraic equation:
\begin{equation}
	\label{eq:critical_points_equation_exponentiated}
	\frac{w\ssp e^{\gamma (\tau+1)}}{w-1}\cdot\frac{w e^{\gamma\rho}-1}{w e^{\gamma(\rho+\tau)}-1}
	\prod_{i=1}^{L}\frac{w\ssp e^{\gamma(a_i+C_i)}-1}{w\ssp e^{\gamma(a_{i+1}+C_i)}-1}
	=1.
\end{equation}

\begin{lemma}
	\label{lemma:critical_points_count_upper_half}
	For any $\tau,\rho>0$,
	equation
	\eqref{eq:critical_points_equation_exponentiated}
	has at most one non-real root in $w$ in the complex upper half-plane.
\end{lemma}
We denote this root in the upper half-plane by $w_c=w_c(\tau,\rho)$.
\begin{proof}[Proof of \Cref{lemma:critical_points_count_upper_half}]
	Denote by 
	$p_{num}(w)$ and
	$p_{den}(w)$
	the polynomials in the denominator and the numerator,
	respectively, in the left-hand side of \eqref{eq:critical_points_equation_exponentiated}.
	Let us first count the real roots of \eqref{eq:critical_points_equation_exponentiated}
	by considering intersections of the graphs of 
	$p_{num}(w)$ and $p_{den}(w)$, $w\in \mathbb{R}$.
	Both polynomials 
	$p_{num}$ and $p_{den}$
	are of degree $L+2$, and have only real roots.
	Since their top degree coefficients
	have the same sign, we may and will
	assume that $p_{num}(-\infty)=p_{den}(-\infty)=+\infty$.
	
	The roots of $p_{num}(w)$ are $0$, $e^{-\gamma \rho}$,
	and $w^{num}_i \coloneqq e^{-\gamma(a_i + C_i)}$, $1\le i\le L$.
	Similarly, $p_{den}(w)$ has 
	roots $1, e^{-\gamma (\rho + \tau)}$, and
	$w^{den}_i \coloneqq e^{-\gamma (a_{i+1}+C_i)}$, $1\le i\le L$.
	By \eqref{eq:scaling_initial_conf}, the roots interlace as
	\begin{equation}
			\label{eq:L_interlacing_roots}
			0 \leq w_{den}^i < w_{num}^i <
			w_{den}^{i-1} \leq 1,
			\qquad 
			2\le i\le L.
	\end{equation} 

	We first discuss two examples illustrating how we count the
	roots. Let us start with $\rho > (a_{L+1}
	+ C_L)$. Then the two leftmost roots of $p_{den}(w)$ are $0$ and
	$e^{-\rho}$, and the two leftmost roots of $p_{num}(w)$ are
	$e^{-\gamma (\rho+\tau)}$ and $e^{-\gamma (a_{L+1} + C_L)}$.
	Moreover, $0<e^{-\gamma (\rho+\tau)}<e^{-\rho}<e^{-\gamma
	(a_{L+1} + C_L)}$.  We see that on each of $L+1$ segments
	between the roots of $p_{num}$, the graph of $p_{num}$
	intersects with the graph of $p_{den}$. This counting
	produces at least $L+1$ real solutions to
	\eqref{eq:critical_points_equation_exponentiated}. Since
	this equation is equivalent to a polynomial equation of
	degree at most $L+2$, it follows that there are no complex
	solutions to
	\eqref{eq:critical_points_equation_exponentiated}. See
	\Cref{fig:roots_case_no_complex} for an illustration.
	
	\begin{figure}[htb]
		\centering
		\includegraphics[width=\textwidth]{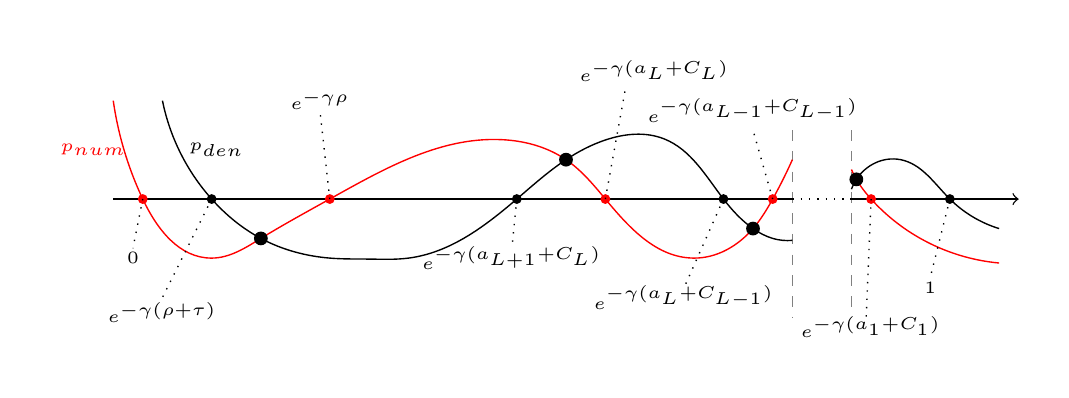}
		\caption{Graphs of $p_{num}$ and $p_{den}$
			for large $\rho$.}
		\label{fig:roots_case_no_complex}
	\end{figure}

	Let us now decrease $\rho$ while keeping $\rho + \tau$
	constant. At some point we will have a repeated root
	$e^{-\gamma \rho} = e^{-\gamma (a_{L+1} + C_L)}$, which upon
	further decreasing $\rho$ breaks the interlacing.  Then we
	can have $L-1$ or $L+1$ roots (counted with multiplicity) at
	the intersections of the graphs of $p_{num}$ and $p_{den}$,
	see \Cref{fig:roots_case_complex_exist} for an illustration.
	When there are $L-1$ intersections,
	\eqref{eq:critical_points_equation_exponentiated} may have a
	single pair of complex conjugate non-real roots. We see that
	there cannot be more than one such pair.

	Now let us describe what happens for $\rho <(a_{L+1} +
	C_L)$.  Recall that $e^{-\gamma (a_{L+1} + C_L)} \le
	e^{-\gamma \rho} \le 1$, so for some $1 \leq i \leq L$ we
	have $w^{den}_i \le e^{-\gamma \rho} \le w^{den}_{i-1}$,
	where we set $w^{den}_{0} = 1$ for convenience.  From
	$\eqref{eq:L_interlacing_roots}$ we also have $w^{den}_i <
	w^{num}_i < w^{den}_{i-1}$, thus we have two roots of the
	numerator, namely $w^{num}_i$ and $e^{-\gamma \rho}$ located
	between two roots of denominator.  Denote the interval
	between $w^{num}_i$ and $e^{-\gamma \rho}$ by $I$, so
	$[0,1]=[0,
	\min(w^{num}_i, e^{-\gamma \rho})]\cup I\cup [\max(w^{num}_i, e^{-\gamma \rho}), 1]$. 
	Note that some of these
	intervals might be empty in the presence of
	double roots of $p_{num}$.
	
	If $e^{-\gamma (\rho + \tau)} \in I$, then the interlacing
	is restored, and we have a structure similar to the first
	case as in \Cref{fig:roots_case_no_complex}. The graphs of
	$p_{num}$ and $p_{den}$	
	intersect $L+1$
	times, which gives at least $L+1$ real critical points, 
	and thus no complex roots exist.

	\begin{figure}[h]
		\centering
		\includegraphics[width=\textwidth]{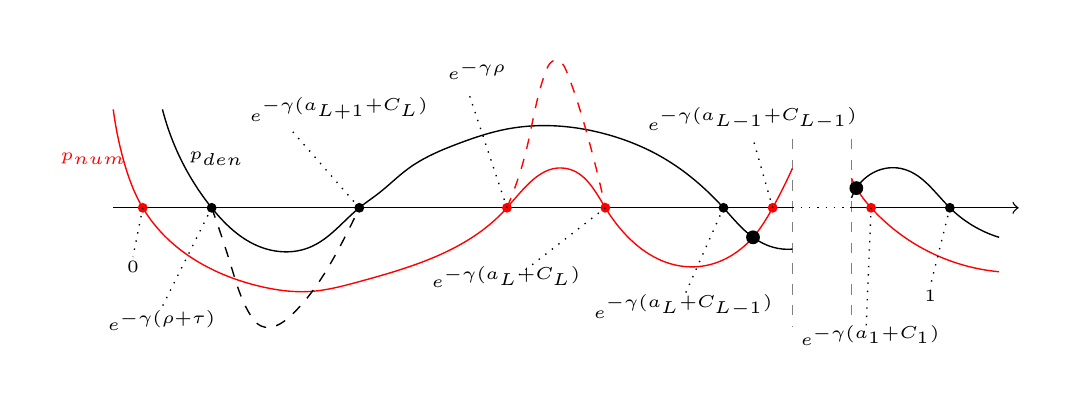}
		\caption{Graphs of $p_{num}$ and $p_{den}$
		for $\rho +\tau > a_{L+1} + c_L > \rho > a_{L}
		+ c_L$.
		Possible dashed graphs lead to 
		nonexistence of complex roots, and 
		solid graphs lead to exactly one complex root in the 
		upper half-plane.}
		\label{fig:roots_case_complex_exist}
	\end{figure}

	On the other hand, if $e^{-\gamma (\rho + \tau)} \notin I$,
	the configuration is similar to
	\Cref{fig:roots_case_complex_exist}, and we have either
	$L-1$ or $L+1$ intersections, leaving the possibility that at most one complex
	root in the upper half-plane exists. This completes the proof.
\end{proof}

\begin{definition}
	\label{def:frozen_boundary}
	Let $\mathcal{D} \subset \mathcal{P}$ be the open set of
	pairs $(\tau, \rho)$, such that $S(w; \tau, \rho)$ defined
	by \eqref{eq:def_S_new}, \eqref{eq:S_via_Li2}
	has one non-real
	critical point $w_c$ in the upper half-plane.
	We call $\mathcal{D}$ \textit{liquid region}, and its
	boundary curve $\partial\mathcal{D}$ the \textit{frozen
	boundary}.
\end{definition}

Let us obtain a parametrization
of the frozen boundary $\partial\mathcal{D}$.
Because the equation \eqref{eq:critical_points_equation_exponentiated}
has real coefficients, 
as $(\tau,\rho)$ approaches $\partial\mathcal{D}$, 
the corresponding critical point $w_c$ becomes close with
its complex conjugate $\overline{w}_c$. At $\partial\mathcal{D}$
these two roots of \eqref{eq:critical_points_equation_exponentiated} 
merge, and thus the frozen boundary is the discriminant curve of the 
equation \eqref{eq:critical_points_equation_exponentiated}. 
We may thus take
$w_c\in \mathbb{R}$ as a parameter of this curve 
$\tau=\tau(w_c)$, $\rho=\rho(w_c)$. 

Denote
\begin{equation*}
	F(w)\coloneqq
	\frac{w}{w-1}
	\prod_{i=1}^{L}\frac{w \ssp e^{\gamma(a_i+C_i)}-1}{w\ssp e^{\gamma(a_{i+1}+C_i)}-1}
	.
\end{equation*}
Then the two equations for the double roots 
of \eqref{eq:critical_points_equation_exponentiated}
yield a rational parametrization
of $\partial\mathcal{D}$
in the exponential coordinates 
$(e^{\gamma\tau},e^{\gamma\rho})$:
\begin{equation}
	\label{eq:rational_parametrization}
	e^{\gamma\tau(w)}=
	\frac{(w F(w))'-e^{-\gamma}}
	{w F'(w)-F(w)+e^{\gamma } F^2(w)}
	,\qquad 
	e^{\gamma\rho(w)}=
	\frac{e^{\gamma } F'(w)}{e^{\gamma } (w F(w))'-1},
	\qquad w\in \mathbb{R}.
\end{equation}
We used this explicit parametrization to
draw the frozen boundaries in 
\Cref{fig:frozen_boundaries} from \Cref{sub:intro_asymptotics}.

\subsection{Analysis of \texorpdfstring{$S(w;\tau,\rho)$}{S}}
\label{sub:analysis_of_S}

In this subsection we assume that $(\tau,\rho)\in
\mathcal{D}$, and investigate the behavior of the steepest
descent contours $\Im S(w; \tau,\rho) = \Im S(w_c; \tau,\rho)$ started from the critical point $w_c$.
In the next
\Cref{sub:deformation_and_sine} we use this information to
deform the original integration contours in
$K_{\mathrm{walks}}$ \eqref{eq:K_walks_intro} to the
steepest descent ones. This will yield
\Cref{thm:bulk_limit}.

First, we consider the behavior of $\Im S(w;\tau,\rho)$ close to the 
real line, that is, $w=v+\mathbf{i}\varepsilon$, 
$v\in \mathbb{R}$,
and
$\varepsilon>0$ is sufficiently small and fixed.
In $\log w$ and $\mathrm{Li}_2(w)$ entering \eqref{eq:S_via_Li2}
we choose the standard branch of the logarithm which has branch cut along the 
negative real line. Using the integral representation in \eqref{eq:Li2},
we see that $\mathrm{Li}_2(\xi)$ has branch cut along $[1,+\infty)$.

\begin{lemma}
	\label{lemma:ImS_behaviour_close_to_real_line}
	For sufficiently small fixed $\varepsilon>0$, the graph of the function 
	$v\mapsto \Im S(v+\mathbf{i}\varepsilon;\tau,\rho)$,
	$v\in \mathbb{R}$,
	has at most four intersections with any horizontal line.
	If there are four intersections, then the 
	leftmost of these intersections is 
	in a small left neighborhood of zero, and 
	goes to $0$ as $\varepsilon\to 0$.
\end{lemma}
See \Cref{fig:Im_graph} for an illustration of the graph of this function.
\begin{proof}[Proof of \Cref{lemma:ImS_behaviour_close_to_real_line}]
	Observe the following behavior of the functions entering \eqref{eq:S_via_Li2}:
	\begin{enumerate}[$\bullet$]
		\item The graph of 
			\begin{equation*}
				v\mapsto \Im(\tau \log(v+\mathbf{i} \varepsilon))=\tau
				\tan^{-1}(\varepsilon/v)+\pi \tau \mathbf{1}_{v<0}
			\end{equation*}
			is in an $O(\varepsilon)$ neighborhood of the graph of the step function
			$v\mapsto \pi\tau\mathbf{1}_{v<0}$ with the added vertical segment from 
			$(0,\pi)$ to $(0,0)$.
		\item 
			The graph of
			$v\mapsto 
				\Im\bigl(
					\mathrm{Li}_2(v+\mathbf{i}\varepsilon) 
			\bigr)$
			is in an $O(\varepsilon)$ neighborhood
			of the graph of the function 
			$\chi_+(v)\coloneqq \pi\log v\cdot \mathbf{1}_{v>1}$. Indeed, this is because
			\begin{equation}
				\label{eq:Im_Li2_computation}
				\begin{split}
					\Im\bigl(\mathrm{Li}_2(v+\mathbf{i} \varepsilon)\bigr)
					&=-\int_0^\infty \mathop{\mathrm{Arg}}
					\left( 1-(v+\mathbf{i}\varepsilon)e^{-u} \right)du
					\\&=
					\pi\int_0^{\infty}\mathbf{1}_{1-ve^{-u}<0}\ssp du+O(\varepsilon)
					\\&=
					\pi \log v\cdot \mathbf{1}_{v>1}+O(\varepsilon).
				\end{split}
			\end{equation}
		\item 
			The graph of $v\mapsto \Im\bigl(
				\mathrm{Li}_2\left( 1 /(v+\mathbf{i}\varepsilon) \right)
			\bigr)$ is in an $O(\varepsilon)$ 
			neighborhood of the graph of the function
			$\chi_-(v)\coloneqq \pi\log v\cdot \mathbf{1}_{0<v<1}$
			with the added vertical line from $(0,0)$ to $(0,-\infty)$.
			This fact is obtained similarly to the expansion 
			\eqref{eq:Im_Li2_computation}.
	\end{enumerate}

	\begin{figure}[htpb]
		\centering
		\includegraphics[width=.6\textwidth]{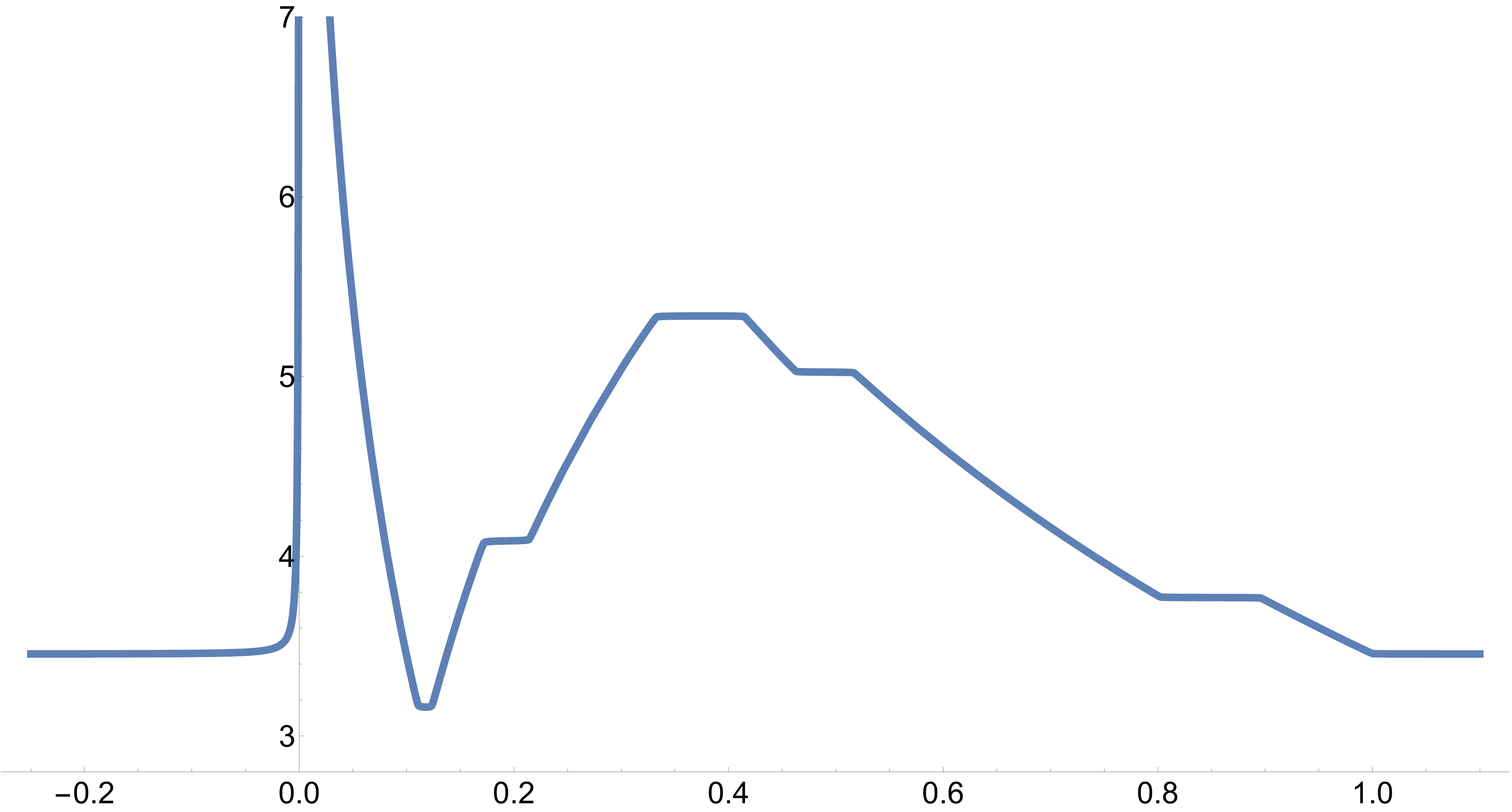}
		\caption{The graph of $v\mapsto \Im S(v+\mathbf{i}\varepsilon;\tau,\rho)$
		for small $\varepsilon$. For any fixed $\varepsilon>0$ this function 
		is continuous. Its maximum is in the neighborhood of zero and 
		has order $O(\left|\log \varepsilon\right|)$.}
		\label{fig:Im_graph}
	\end{figure}

	Thus, for small $\varepsilon$ the 
	graph of 
	$v\mapsto \Im S(v+\mathbf{i}\varepsilon;\tau,\rho)$
	belongs to an $O(\varepsilon)$ neighborhood 
	of the graph of the following function:
	\begin{equation*}
		\begin{split}
		&
		S_{\mathbb{R}}(v;\tau,\rho)\coloneqq
		\pi\tau\mathbf{1}_{v<0}
		-
		\gamma^{-1}
		\chi_+( ve^{\gamma\rho} )
		+
		\gamma^{-1}
		\chi_+( ve^{\gamma(\rho+\tau)} )
		\\&\hspace{150pt}
		-
		\gamma^{-1}
		\chi_-( v )
		+
		\gamma^{-1}
		\sum_{i=1}^L
		\left[ 
			\chi_-(v e^{\gamma(a_i+C_i)}     )-
			\chi_-(v e^{\gamma(a_{i+1}+C_i)} )
		\right].
	\end{split}
	\end{equation*}
	We see that 
	for $v<0$ and for sufficiently large $v$, 
	the function $S_{\mathbb{R}}$ is equal to $\pi\ssp\tau$.
	Next, the function $S_{\mathbb{R}}$
	is piecewise linear in $\log v$. Due to the 
	ordering of $a_i$ and $C_i$ \eqref{eq:a_C_relation}, 
	one readily sees that $S_{\mathbb{R}}(v;\tau,\rho)$
	for $v>0$ first weakly decreases in $v$, then it may 
	weakly increase $v$, and finally it weakly decreases in $v$ 
	again
	until it stabilizes at the value $\pi\tau$.

	Moreover, for any fixed $\varepsilon>0$, the pre-limit function
	$v\mapsto \Im S(v+\mathbf{i}\varepsilon;\tau,\rho)$
	is not constant. Thus, we see that the
	graph of the pre-limit function may intersect any horizontal line
	at most four times: at most once in a small left neighborhood of
	$v=0$, and at most three times
	for $v>0$. 
	For small $\varepsilon$, 
	the graph of 
	$v\mapsto \Im S(v+\mathbf{i}\varepsilon;\tau,\rho)$ 
	becomes more and more vertical, and 
	thus we see that the leftmost point 
	of intersection with a horizontal line
	goes to $0$ as $\varepsilon\to 0$.
	This completes the proof.
\end{proof}

Let us now look at the behavior of $\Im S(w;\tau,\rho)$
for large $|w|$.

\begin{lemma}
	\label{lemma:ImS_for_large_w}
	We have
	\begin{equation*}
		\lim_{R\to+\infty}
		\Im S(Re^{\mathbf{i}\theta};\tau,\rho)
		=\pi\tau,
	\end{equation*}
	uniformly in $\theta\in [0,\pi]$.
\end{lemma}
\begin{proof}
	Clearly, we have
	$\Im(\log Re^{\mathbf{i}\theta})=\theta$.
	Moreover, 
	$\mathrm{Li}_2(w^{-1})\to 0$ for $|w|\to+\infty$
	because $\mathrm{Li}_2$ is continuous at $0$.
	To complete the proof, it remains to show that
	\begin{equation*}
		\Im\bigl(\mathrm{Li}_2(R e^{\mathbf{i}\theta})\bigr)\sim
		(\pi-\theta)\log R,\qquad R\to+\infty,
	\end{equation*}
	uniformly in $\theta\in[0,\pi]$.
	We have
	\cite[(25.12.4)]{NIST:DLMF}
	\begin{equation*}
		\Im\bigl(\mathrm{Li}_2(R e^{\mathbf{i}\theta})\bigr)
		=
		-
		\Im\bigl(\mathrm{Li}_2(R^{-1} e^{-\mathbf{i}\theta})\bigr)
		-\frac{1}{2}\ssp \Im
		\left( \log(-R e^{\mathbf{i}\theta}) \right)^2.
	\end{equation*}
	The first term goes to zero as $R\to+\infty$, and for the second 
	term we have
	\begin{equation*}
		-\frac{1}{2}\ssp \Im
		\left( \log(-R e^{\mathbf{i}\theta}) \right)^2=
		-\Re\left( \log(-R e^{\mathbf{i}\theta}) \right)
		\Im \left( \log(-R e^{\mathbf{i}\theta}) \right)
		=-\log R\cdot(\theta-\pi),
	\end{equation*}
	and we are done.
\end{proof}

\subsection{Contour deformation and convergence to the incomplete beta kernel}
\label{sub:deformation_and_sine}

\Cref{lemma:ImS_behaviour_close_to_real_line,lemma:ImS_for_large_w}
imply that when $(\tau,\rho)$ is in the liquid region $\mathcal{D}$,
all four half-contours emanating from the 
critical point $w_c$ in the upper half plane
end on the real line. Let us denote these points
by
\begin{equation}
	\label{eq:4_points_of_intersection}
	u^z_- = 0 < u^w_- < u^z_+ < u^w_+.
\end{equation}
We take the leftmost point to be
$0$ as the $\varepsilon\to 0$ limit of the 
leftmost point of intersection in 
\Cref{lemma:ImS_behaviour_close_to_real_line}.
Moreover, from the proof of \Cref{lemma:ImS_behaviour_close_to_real_line}
we see that
\begin{equation}
	\label{eq:u_z_u_w_ordering}
	u^w_-< e^{-\gamma(\tau+\rho)} < u^z_+ < e^{-\gamma\rho} < u^w_+.
\end{equation}

Let us denote two closed, positively 
oriented contours 
$\Im S(w; \tau,\rho) = \Im S(w_c; \tau,\rho)$
by $\gamma_z$, $\gamma_w$, where both of them pass through
$w_c$ and $\overline{w}_c$, the contour 
$\gamma_z$ goes through 
$u^z_-, u^z_+$, and 
$\gamma_w$ goes through 
$u^w_-, u^w_+$.

\begin{lemma}
	\label{lemma:Re_S_at_intersection_points}
	We have 
	$\Re S(u^z_{\pm};\tau,\rho)
	>
	\Re S(w_c;\tau,\rho)$
	and 
	$\Re S(u^w_{\pm};\tau,\rho)
	<
	\Re S(w_c;\tau,\rho)$.
\end{lemma}
These inequalities justify our notation for the
points \eqref{eq:4_points_of_intersection} and 
the contours $\gamma_z,\gamma_w$. The latter
will be the new integration contours in the 
correlation kernel.
\begin{proof}[Proof of \Cref{lemma:Re_S_at_intersection_points}]
	Clearly, on the contours
	$\Im S(w; \tau,\rho) = \Im S(w_c; \tau,\rho)$
	(which are the
	steepest descent/ascent ones),
	the real part of $S$ is monotone. 
	Moreover, the increasing and decreasing behavior of
	$\Re S$
	alternates throughout
	the four half-contours originating at the 
	critical point $w_c$ (which is a saddle point for
	$\Re S$). Therefore, the result will follow if we show that
	$+\infty=\Re S(u^z_{-};\tau,\rho)
	>
	\Re S(w_c;\tau,\rho)$.
	That is, let us show that
	$\Re S(w,\tau,\rho)\to+\infty$
	as $|w|\to 0$.
	
	Using 
	\cite[(25.12.4)]{NIST:DLMF}
	similarly to the proof of \Cref{lemma:ImS_for_large_w},
	we can write for small $|w|$:
	\begin{equation*}
		\Re \bigl(\mathrm{Li}_2(w^{-1})\bigr)=
		-
		\Re \bigl(\mathrm{Li}_2(w)\bigr)-\frac{\pi^2}{6}-
		\frac{1}{2}\Re
		\left( \log(-w^{-1}) \right)^2
	\end{equation*}
	For $|w|\to0$, 
	the first summand in the right-hand side goes to zero,
	while for the last one we have
	\begin{equation*}
		-
		\frac{1}{2}\Re
		\left( \log(-w^{-1}) \right)^2
		=-\frac{1}{2}(\log|w|)^2
		+\frac{1}{2}(\mathrm{Arg}(-w^{-1}))^2
	\end{equation*}
	The argument is bounded, and so we see using \eqref{eq:S_via_Li2}
	that $\Re S(w;\tau,\rho)$ behaves for small $|w|$ as
	\begin{equation*}
		-\tau\log (|w|^{-1})+
		\frac{\gamma^{-1}}{2}(\log|w|)^2
		+\frac{\gamma^{-1}}{2}
		\sum_{i=1}^L
		\left( \bigl(\log\bigl|w e^{\gamma(a_{i+1}+C_i)}\bigr|\bigr)^2
		-
		\bigl(\log\bigl|w e^{\gamma(a_{i}+C_i)}\bigr|\bigr)^2
		\right)+\mathrm{const}.
	\end{equation*}
	The term $-\tau\log(|w|^{-1})$ is of smaller order than 
	the squared logarithms,
	and one readily sees that the total contribution of the latter 
	is $+\infty$. This completes the proof.
\end{proof}

Let us recall the
original integration contours in the kernel
$K_{\mathrm{walks}}$ \eqref{eq:K_walks_intro} which
we reproduce here for convenience:
\begin{equation}
	\label{eq:K_walks_for_deformation}
		\begin{split}
			&K_{\mathrm{walks}}(p_1,t_1;p_2,t_2)
			=
			\mathbf{1}_{t_1=t_2}\mathbf{1}_{p_1=p_2}-
			\mathbf{1}_{t_2>t_1}\mathbf{1}_{p_2+t_2>p_1+t_1}
			\frac{q^{(t_1-t_2)(p_1+t_1)}(q^{p_1-p_2+t_1-t_2+1};q)_{t_2-t_1-1}}{(q;q)_{t_2-t_1-1}}
			\\&\hspace{20pt}
			-\frac{q^{-t_1-p_1}}{(2\pi\mathbf{i})^2}
			\oiint
			dz\ssp dw\ssp
			\frac{z^{-t_2}w^{t_1} }{w-z}
			\frac{(q;q)_{t_1}}{(wq^{-p_1-t_1};q)_{t_1+1}}
			\frac{(zq^{1-p_2-t_2};q)_{t_2-1}}{(q;q)_{t_2-1}}
			\frac{(w^{-1};q)_{\infty}}{(z^{-1};q)_{\infty}}
			\prod_{r=1}^{m}\frac{1-q^{x_r}/z}{1-q^{x_r}/w},
		\end{split}
\end{equation}
The $w$ contour is an arbitrarily small positively oriented 
circle around $0$, and the 
$z$ contour is positively oriented,
goes around $q^{p_2+t_2},q^{p_2+t_2+1},q^{p_2+t_2+2},\ldots$ and
the $w$ contour, and encircles no other $z$ poles of the integrand. 
The singularities of the integrand are as follows
(see \Cref{fig:original_contours} for an illustration):
\begin{enumerate}[$\bullet$]
	\item In $w$, there is
		an essential singularity at $w=0$, and all the simple poles are 
		at
		\begin{equation}
			\label{eq:w_poles_integrand}
			w=z
			\quad \textnormal{and}
			\quad 
			w\in \bigl\{ q^{p_1+j} \bigr\}_{j=0}^{t_1} \bigcap \ssp \bigl\{q^{x_r} \bigr\}_{r=1}^m.
		\end{equation}
	\item In $z$, all the simple poles are at
		\begin{equation}
			\label{eq:z_poles_integrand}
			z=w
			\quad \textnormal{and}
			\quad 
			z\in \bigl\{ q^j \bigr\}_{j=0}^\infty
			\setminus
			\left( \bigr\{ q^{p_2+j} \bigr\}_{j=1}^{t_2-1} 
			\bigcup \ssp \bigl\{ q^{x_r} \bigr\}_{r=1}^{m} \right).
		\end{equation}
\end{enumerate}
Note that $z=0$ is not a pole thanks to the presence of 
the function $(z^{-1};q)_{\infty}$ in the denominator.
Moreover, observe that at infinity, the integrand behaves
as $O(w^{-2})$ as a function of $w$. This implies that it has no 
residue at $w=\infty$.

\begin{figure}[htpb]
	\centering
	\includegraphics[width=.8\textwidth]{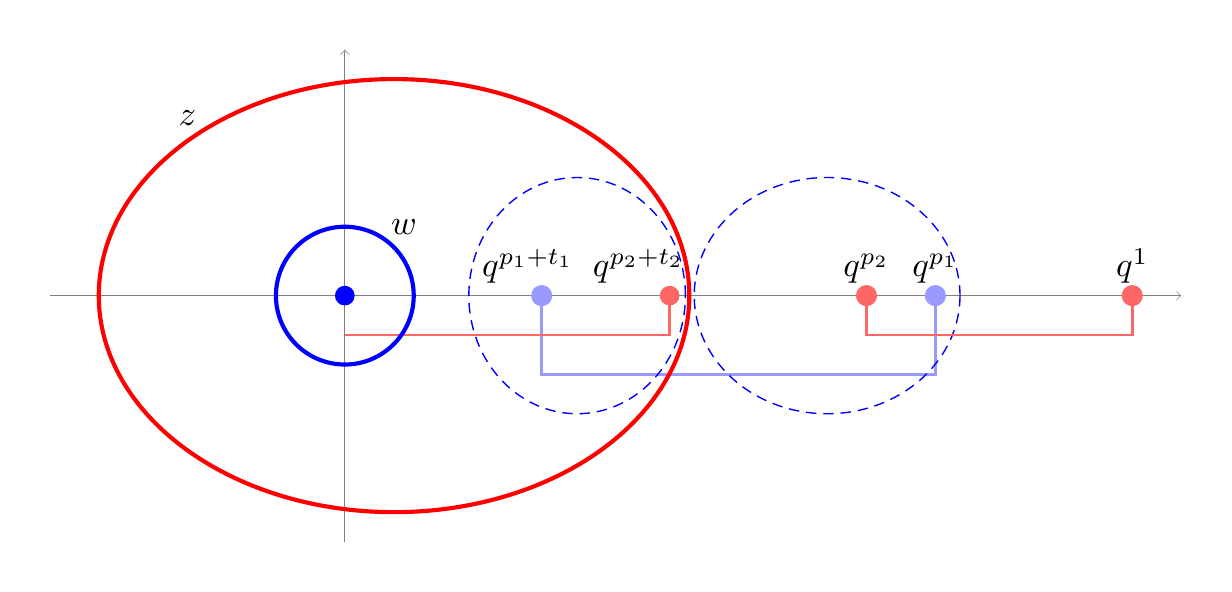}
	\caption{Thick curves represent the original $w$ and $z$ contours 
	in $K_{\mathrm{walks}}$ \eqref{eq:K_walks_for_deformation}.
	The possible $w$ poles \eqref{eq:w_poles_integrand} lie between $q^{p_1}$
	and $q^{p_1+t_1}$. The possible $z$ poles \eqref{eq:z_poles_integrand}
	lie outside of the segment between $q^{p_2+1}$ and $q^{p_2+t_2-1}$.
	Note that the relative positions of $p_1$ and $p_2$, as well as of 
	$p_1+t_1$ and $p_2+t_2$, may be arbitrary, and in the
	figure we display only one such possibility.
	The union of the dashed curves is the new $w$ contour after we drag it through infinity.}
	\label{fig:original_contours}
\end{figure}

In the bulk asymptotic regime 
\eqref{eq:asymptotic_regime}--\eqref{eq:scaling_initial_conf},
assume that the position 
$(\tau,\rho)$ is in the liquid region $\mathcal{D}$ (\Cref{def:frozen_boundary}).
We aim to deform the contours 
in $K_{\mathrm{walks}}$ \eqref{eq:K_walks_for_deformation}
to new contours which intersect at the non-real critical 
points $w_c, \overline{w}_c$, and coincide with the steepest descent contours
$\gamma_z,\gamma_w$ (defined before \Cref{lemma:Re_S_at_intersection_points}) 
outside a small neighborhood
of the real line.
Fix small $\varepsilon>0$, and perform the contour deformations in the following
order:
\begin{enumerate}[(1)\/]
	\item 
		Keeping the $w$ contour a small circle around $0$ 
		of radius $\varepsilon/2$, deform the $z$
		contour to coincide with the steepest descent contour $\gamma_z$
		outside of the $\varepsilon$-neighborhood of $\mathbb{R}$.
		In the $\varepsilon$-neighborhood of $\mathbb{R}$,
		we need to make sure that the deformation from the
		old to the new $z$ contour does not cross any $z$-poles
		of the integrand.
		Namely, 
		in the $\varepsilon$-neighborhood of $u^z_-=0$, let the new $z$ contour 
		pass around $0$ following a circle of radius $\varepsilon$ instead
		of going straight to $0$ along $\gamma_z$.
		Around $u^z_+$ which is between $e^{-\gamma(\tau+\rho)}$ and $e^{-\gamma\rho}$
		(see \eqref{eq:u_z_u_w_ordering})
		but may not be between $q^{p_2+t_2}$ and $q^{p_2}$,
		let the new $z$ contour follow straight lines at distance $\varepsilon$ from
		$\mathbb{R}$, and then go around the existing poles 
		at distance at least $\varepsilon$ 
		from these poles 
		(see \Cref{fig:new_z_contour_epsilon} for an illustration).
		Denote the new $z$ contour by $\gamma_z^\varepsilon$.

	\item Drag $w$ through infinity, that is, replace the $w$ integral 
		over a small contour around $0$ by minus the integral 
		over all the other $w$-poles which are listed in \eqref{eq:w_poles_integrand}.
		Thus, we obtain minus the integral over the union of the dashed contours
		in \Cref{fig:original_contours}, 
		minus $2\pi\mathbf{i}$ times 
		the residue of the integrand at $w=z$ (which is still under the single 
		integral in $z$ over $\gamma_z^\varepsilon$).

	\item 
		Now let us deform the $w$ contour 
		to the steepest descent contour $\gamma_w$
		outside the $\varepsilon$-neighborhood of $\mathbb{R}$.
		In the $\varepsilon$-neighborhood of $\mathbb{R}$
		let us
		modify the new $w$ contour so that the deformation does not pick any 
		residues one the real line, at points $w=q^j$ (this is done similarly to the
		contour $\gamma_z^\varepsilon$, see \Cref{fig:new_z_contour_epsilon}
		for an illustration).
		Denote the new $w$ contour by $\gamma_w^\varepsilon$.
		The deformation of the $w$ contour to $\gamma_w^\varepsilon$
		picks a residue at $w=z$ if $z$ is in the right
		part of its contour, from $\overline w_c$ to $w_c$
		in the counterclockwise order.
\end{enumerate}

Accounting for all residues and sign changes 
throughout the contour deformation,
we see that the kernel $K_{\mathrm{walks}}$ 
\eqref{eq:K_walks_for_deformation} 
takes the form:
\begin{equation}
	\label{eq:K_walks_deformed_contours}
		\begin{split}
			&K_{\mathrm{walks}}(p_1,t_1;p_2,t_2)
			=
			-
			\mathbf{1}_{t_2>t_1}\mathbf{1}_{p_2+t_2>p_1+t_1}
			\frac{q^{(t_1-t_2)(p_1+t_1)}(q^{p_1-p_2+t_1-t_2+1};q)_{t_2-t_1-1}}
			{(q;q)_{t_2-t_1-1}}
			\\&\hspace{20pt}
			+
			\mathbf{1}_{t_1=t_2}\mathbf{1}_{p_1=p_2}
			+
			\frac{q^{-t_1-p_1}}{2\pi\mathbf{i}}
			\int_{w_c\to \overline w_c}
			\frac{
			z^{t_1-t_2}
			(q;q)_{t_1}(zq^{1-p_2-t_2};q)_{t_2-1}}
			{(q;q)_{t_2-1}(zq^{-p_1-t_1};q)_{t_1+1}}
			\ssp dz
			\\&\hspace{20pt}
			+
			\frac{q^{-t_1-p_1}}{(2\pi\mathbf{i})^2}
			\oint_{\gamma_z^\varepsilon}dz \oint_{\gamma_w^\varepsilon}dw\ssp
			\frac{(q;q)_{t_1}}{(q;q)_{t_2-1}}
			\frac{\exp\left\{m \left(  
			S_m(w;t_1,p_1)-S_m(z;t_2,p_2) 
			\right) 
			\right\}}{(w-z)(1-zq^{-p_2})(1-zq^{-p_2-t_2})},
		\end{split}
\end{equation}
Here the single integral is over the \emph{left}
part of the contour $\gamma_z^\varepsilon$, 
from $w_c$ to $\overline w_c$ in the counterclockwise order,
and we used the notation \eqref{eq:def_S_m_new}.

\begin{figure}[htpb]
	\centering
	\includegraphics[width=.8\textwidth]{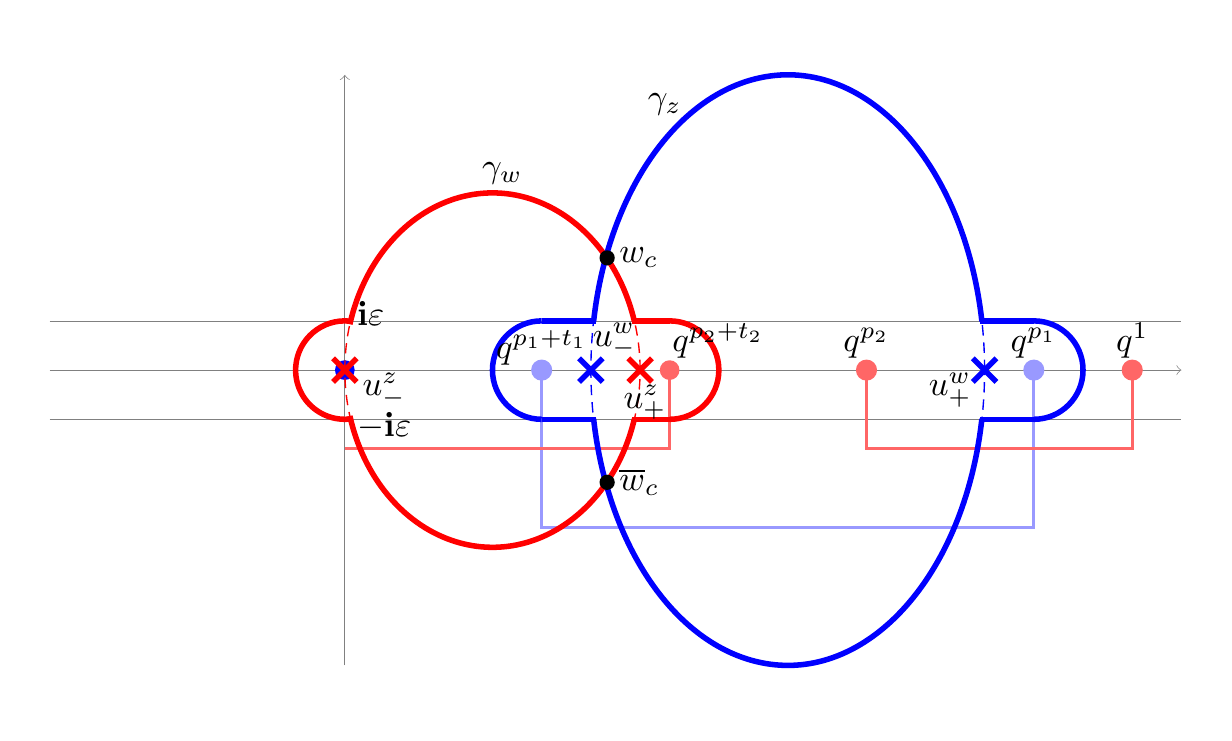}
	\caption{Deformed integration contours to the steepest descent ones,
	with modifications in the $\varepsilon$-neighborhood
	of the real line to avoid picking unnecessary residues.
	Note that for large $m$, not all four modifications
	in the $\varepsilon$-neighborhood are present.}
	\label{fig:new_z_contour_epsilon}
\end{figure}

\begin{lemma}
	\label{lemma:double_goes_to_zero}
	With $\varepsilon=m^{-1}$,
	in the bulk limit regime
	\eqref{eq:asymptotic_regime}--\eqref{eq:scaling_initial_conf},
	the double contour integral in 
	\eqref{eq:K_walks_deformed_contours} goes to zero.
\end{lemma}
\begin{proof}
	All the quantities except 
	$\exp\left\{m \left(  
	S_m(w;t_1,p_1)-S_m(z;t_2,p_2) 
	\right) 
	\right\}$
	in the double contour integral 
	stay bounded in our limit regime. 
	By \Cref{lemma:S_convergence}, the functions $S_m$ are 
	well-approximated by $S$.
	Since the integration contours are steepest descent for 
	$S$ outside the $\varepsilon$-neighborhood of $\mathbb{R}$,
	we see that the contribution from the 
	parts of the contours 
	away from the real line 
	goes to zero. This is because outside a small neighborhood of 
	$w_c$ and $\overline w_c$, the integrand 
	is bounded in absolute value by $e^{-cm}$ for some $c>0$.

	To estimate the contribution from the neighborhood of the real line,
	we need to bound the derivative of
	$\Re S(w;\tau,\rho)$ along the straight and the circular parts of the additional
	contours. All of these non-steepest descent additional contours have 
	length of order $\varepsilon=m^{-1}$.
	Indeed, for example, an additional contour may go from
	$q^{p_1}$ to $u_+^w$, and because $u_+^w>e^{-\gamma\rho}$,
	the length of this contour is bounded from above by the 
	distance between $q^{p_1}$ and $e^{-\gamma\rho}$.
	This distance is of order $1/m$, see \eqref{eq:asymptotic_regime}.
	Thus,
	it suffices to bound
	the derivative of $\Re S(w;\tau,\rho)$ on the additional contours
	by $o(m^{1-\delta})$ for some $\delta>0$.
	Indeed, then the total change of
	$\Re S(w;\tau,\rho)$ along the non-steepest descent additional contours
	is of order $o(m^{-\delta})$, and $e^{m ( -c+o(m^{-\delta}) )}$
	still goes to zero exponentially fast.

	To estimate the derivative of the real part of a function
	$f(z)=u(x,y)+iv(x,y)$ which is holomorphic in a neighborhood of a
	curve $z(\theta)=(x(\theta),y(\theta))$, we
	have by the Cauchy--Riemann equations:
	\begin{equation*}
		\frac{\partial}{\partial \theta}\Re f(z(\theta))=
		x'(\theta)\Re f'(z(\theta))-
		y'(\theta)\Im f'(z(\theta)).
	\end{equation*}
	
	Let us now turn to the function
	$S(w;\tau,\rho)$ \eqref{eq:S_via_Li2}.
	Different summands in \eqref{eq:S_via_Li2}
	have different singularities, 
	let us consider each of these singularities in order.
	First, in a neighborhood of
	$u^z_-=0$ the modified contour 
	$w(\theta)=\varepsilon e^{\mathbf{i}\theta}$
	goes in a circular way 
	without straight
	parts. We have
	(here and below in the proof, $C$ denotes a fixed sufficiently large 
	positive constant whose value may differ from one inequality
	to the next):
	\begin{equation*}
		\frac{\partial}{\partial \theta}\Re\bigl(\log w(\theta)\bigr)=0,
		\qquad 
		\left|
		\frac{\partial}{\partial \theta}\Re\bigl(\mathrm{Li}_2(Aw(\theta)^{-1})\bigr)
		\right|\le  C,
	\end{equation*}
	for any $A>0$, and all other summands in \eqref{eq:S_via_Li2} are 
	regular around $0$. 

	The next singularities may appear in the neighborhoods
	of $u^w_{\pm}$ or $u^z_+$. In these neighborhoods, 
	the modified contours may contain straight lines and circular
	segments. By changing variables, it suffices to estimate only 
	the derivatives of the real parts of
	$\mathrm{Li}_2(w)$ and $\mathrm{Li}_2(w^{-1})$
	in the neighborhood of $w=1$ (at all other points except $w=0$ 
	these functions 
	are regular, and we already considered $w=0$ above in the proof).
	We have for the circular contours $w(\theta)=
	1+\varepsilon e^{\mathbf{i}\theta}$:
	\begin{equation*}
		\left|\frac{\partial}{\partial \theta}\Re 
		\bigl(
		\mathrm{Li}_2(w(\theta))\bigr)
		\right|
		\le C \varepsilon\log(\varepsilon^{-1}),
		\qquad 
		\left|\frac{\partial}{\partial \theta}\Re 
		\bigl(
		\mathrm{Li}_2(w(\theta)^{-1})
		\bigr)
		\right|
		\le C \varepsilon\log(\varepsilon^{-1}),
	\end{equation*}
	For the straight contours $w(x)=x\pm\mathbf{i} \varepsilon$ we have:
	\begin{equation*}
		\left|\frac{\partial}{\partial x}\Re 
		\bigl(
		\mathrm{Li}_2(w(x))\bigr)
		\right|
		\le C \log(\varepsilon^{-1}),
		\qquad 
		\left|\frac{\partial}{\partial x}\Re 
		\bigl(
		\mathrm{Li}_2(w(x)^{-1})
		\bigr)
		\right|
		\le C \log(\varepsilon^{-1}),
	\end{equation*}
	We see that the derivative of $\Re S(w;\tau,\rho)$
	is upper bounded (in the absolute value)
	by $C\log(\varepsilon^{-1})=C\log m$, which is $o(m^{1-\delta})$
	for any $\delta<1$.
	This completes the proof.
\end{proof}

It remains to compute the limit of all the other terms in the right-hand
side of \eqref{eq:K_walks_deformed_contours}
except the negligible double integral:

\begin{lemma}
	\label{lemma:single_integrals}
	In the bulk limit regime
	\eqref{eq:asymptotic_regime}--\eqref{eq:scaling_initial_conf},
	the sum of the first three terms in
	\eqref{eq:K_walks_deformed_contours}
	converges to
	\begin{equation*}
		(-1)^{\Delta t}
		e^{-\gamma(\tau+\rho)\Delta t}
		\left( 
		\mathbf{1}_{t_1=t_2}\mathbf{1}_{p_1=p_2}
		-
		\mathsf{B}_{\omega}(t_1-t_2,p_1-p_2)
		\right),
	\end{equation*}
	where $\mathsf{B}_{\omega}$ is the incomplete beta
	kernel
	(\Cref{def:incomplete_beta_kernel}),
	and 
	\begin{equation}
		\label{eq:omega_for_beta_defn_in_last_step}
		\omega=\omega(\tau,\rho)
		\coloneqq
		\frac{1-w_c(\tau,\rho) e^{\gamma \rho}}{1-w_c(\tau,\rho) e^{\gamma(\tau+\rho)}},
	\end{equation}
	where $w_c(\tau,\rho)$
	is the critical point of $S(w;\tau,\rho)$
	\eqref{eq:S_via_Li2}
	in the upper half-plane 
	(see \Cref{lemma:critical_points_count_upper_half}).
\end{lemma}
The factor 
$(-1)^{\Delta t}
e^{-\gamma(\tau+\rho)\Delta t}$
is simply a gauge transformation of the kernel which does not change
a determinantal process.
\begin{proof}[Proof of \Cref{lemma:single_integrals}]
	Recall that the quantities
	$\Delta t=t_1-t_2$, $\Delta p=p_1-p_2$ are fixed.
	The
	first three terms in the
	right-hand side of \eqref{eq:K_walks_deformed_contours}
	have the form
	\begin{equation}
		\label{eq:beta_computation_last_step}
		\begin{split}
			&
			\mathbf{1}_{\Delta t=\Delta p=0}
			-
			\mathbf{1}_{\Delta t<0}\mathbf{1}_{\Delta t+\Delta p<0}
			\frac{q^{\Delta t(p_1+t_1)}(q^{\Delta t+\Delta p+1};q)_{-\Delta t-1}}
			{(q;q)_{-\Delta t-1}}
			\\&\hspace{120pt}
			+
			\frac{q^{-t_1-p_1}}{2\pi\mathbf{i}}
			\int_{w_c\to \overline w_c}
			\frac{
			z^{\Delta t}
			(q;q)_{t_2+\Delta t}(zq^{1-p_2-t_2};q)_{t_2-1}}
			{(q;q)_{t_2-1}(zq^{-p_2-t_2-\Delta p-\Delta t};q)_{t_2+\Delta t+1}}
			\ssp dz.
		\end{split}
	\end{equation}
	Here the integration arc is the left part of the contour,
	from $w_c$ to $\overline w_c$ in the counterclockwise order.

	We have for $\Delta t<0$ and $\Delta t+\Delta p<0$:
	\begin{equation*}
		\frac{q^{\Delta t(p_1+t_1)}(q^{\Delta t+\Delta p+1};q)_{-\Delta t-1}}
			{(q;q)_{-\Delta t-1}}
			\to
			(-1)^{-\Delta t-1}
			e^{-\gamma(\tau+\rho)\Delta t} 
			\binom{-\Delta t-\Delta p-1}{-\Delta t-1}.
	\end{equation*}
	Indeed, this is because $\frac{1-q^a}{1-q^b}\to\frac{a}{b}$ for fixed
	$a,b\in \mathbb{Z}_{\ge1}$
	as $q\to 1$.

	In the integrand, we have
	\begin{equation*}
		\frac{
		(q;q)_{t_2+\Delta t}}
		{(q;q)_{t_2-1}}=(q^{t_2};q)_{\Delta t+1}\to
		(1-e^{-\gamma \tau})^{\Delta t+1}
		,
	\end{equation*}
	using the standard notation of the $q$-Pochhammer
	symbol
	$(a;q)_{-k}=(aq^{-k};q)_{k}^{-1}$, $k\in \mathbb{Z}_{\ge0}$,
	with a negative index.
	Similarly,
	\begin{multline*}
		\frac{
		(zq^{1-p_2-t_2};q)_{t_2-1}}
		{(zq^{-p_2-t_2-\Delta p-\Delta t};q)_{t_2+\Delta t+1}}
		\\=
		\frac{(z q^{1-p_2-\Delta p};q)_{\Delta p-1}}{(zq^{-p_2-t_2-\Delta p-\Delta t};q)_{
		\Delta t+\Delta p+1}}
		\to
		\bigl( 1-z e^{\gamma \rho} \bigr)^{\Delta p-1}
		\bigl( 1-z e^{\gamma(\tau+\rho)} \bigr)^{-\Delta t-\Delta p-1}.
	\end{multline*}
	
	Let us make a change of variables
	\begin{equation*}
		u=\frac{1-z e^{\gamma \rho}}{1-z e^{\gamma(\tau+\rho)}},
		\qquad 
		z=e^{-\gamma \rho}\ssp\frac{1-u}{1-e^{\gamma\tau}u},
		\qquad 
		dz=-e^{-\gamma \rho}
		\frac{1-e^{\gamma \tau}}{(1-e^{\gamma \tau}u)^2}\ssp du.
	\end{equation*}
	With this change of variables, \eqref{eq:beta_computation_last_step}
	converges to 
	\begin{equation}
		\label{eq:beta_computation_last_step2}
		\begin{split}
			\mathbf{1}_{\Delta t=\Delta p=0}
			&+
			\mathbf{1}_{\Delta t<0}\mathbf{1}_{\Delta t+\Delta p<0}
			(-1)^{\Delta t}
			e^{-\gamma(\tau+\rho)\Delta t} 
			\binom{-\Delta t-\Delta p-1}{-\Delta t-1}
			\\&+
			\frac{(-1)^{\Delta t}e^{-\gamma(\tau+\rho)\Delta t}}{2\pi\mathbf{i}}
			\int_{\omega\to\overline\omega}
			(1-u)^{\Delta t}u^{\Delta p-1}\ssp du,
		\end{split}
	\end{equation}
	where $\omega$ is given by 
	\eqref{eq:omega_for_beta_defn_in_last_step},
	and this point is in the upper half-plane.
	The integration arc goes from $\omega$ to $\overline\omega$ and crosses the 
	real line between $0$ and $1$.

	In \eqref{eq:beta_computation_last_step2},
	we can remove the overall factor 
	$(-1)^{\Delta t}
	e^{-\gamma(\tau+\rho)\Delta t}$
	as it is a gauge transformation leading to an equivalent determinantal kernel.
	Finally, for $\Delta t<0$, let us write
	\begin{equation*}
		\frac{1}{2\pi\mathbf{i}}
		\int_{\omega\to\overline\omega}
		(1-u)^{\Delta t}u^{\Delta p-1}\ssp du
		=
		-
		\mathop{\mathrm{Res}}\nolimits_{u=0}
		(1-u)^{\Delta t}u^{\Delta p-1}
		-
		\frac{1}{2\pi\mathbf{i}}
		\int_{\overline\omega}^{\omega}
		(1-u)^{\Delta t}u^{\Delta p-1}\ssp du,
	\end{equation*}
	where the integration arc 
	from $\overline\omega$ to $\omega$
	in the right-hand side crosses the real line to the left of 
	$0$. One readily sees that the minus residue at $u=0$ 
	exactly cancels out with the second summand 
	in \eqref{eq:beta_computation_last_step2}.
	For $\Delta t\ge0$, the integral in \eqref{eq:beta_computation_last_step2}
	is equal to 
	$-
	\frac{1}{2\pi\mathbf{i}}
	\int_{\overline\omega}^{\omega}
	(1-u)^{\Delta t}u^{\Delta p-1}\ssp du$,
	where the integration arc from $\overline\omega$ to $\omega$
	crosses the real line between $0$ and $1$.
	This completes the proof.
\end{proof}

The contour deformations in the kernel
$K_{\mathrm{walks}}$
\eqref{eq:K_walks_for_deformation}
and
\Cref{lemma:single_integrals,lemma:double_goes_to_zero}
complete the proof of
\Cref{thm:bulk_limit}.

\medskip

\textsc{University of Virginia, Charlottesville, VA, USA}

E-mail: \texttt{lenia.petrov@gmail.com}

E-mail: \texttt{me@mtikhonov.com}


\begin{thebibliography}{BKMM07}

\bibitem[Agg19]{aggarwal2019universality}
A.~Aggarwal.
\newblock {Universality for Lozenge Tiling Local Statistics}.
\newblock {\em arXiv preprint}, 2019.
\newblock arXiv:1907.09991 [math.PR].

\bibitem[BG13]{BG2011non}
A.~Borodin and V.~Gorin.
\newblock {Markov processes of infinitely many nonintersecting random walks}.
\newblock {\em Probab. Theory Relat. Fields}, 155(3-4):935--997, 2013.
\newblock arXiv:1106.1299 [math.PR].

\bibitem[BGR10]{borodin-gr2009q}
A.~Borodin, V.~Gorin, and E.~Rains.
\newblock {q-Distributions on boxed plane partitions}.
\newblock {\em Selecta Math.}, 16(4):731--789, 2010.
\newblock arXiv:0905.0679 [math-ph].

\bibitem[BKMM07]{BKMM2003}
J.~Baik, T.~Kriecherbauer, K.~T.-R. McLaughlin, and P.~D. Miller.
\newblock {\em {Discrete Orthogonal Polynomials: Asymptotics and
  Applications}}.
\newblock Annals of Mathematics Studies. Princeton University Press, 2007.
\newblock arXiv:math/0310278 [math.CA].

\bibitem[BOO00]{borodin2000b}
A.~Borodin, A.~Okounkov, and G.~Olshanski.
\newblock {Asymptotics of Plancherel measures for symmetric groups}.
\newblock {\em Jour. AMS}, 13(3):481--515, 2000.
\newblock arXiv:math/9905032 [math.CO].

\bibitem[Bor11]{Borodin2009}
A.~Borodin.
\newblock Determinantal point processes.
\newblock In G.~Akemann, J.~Baik, and P.~Di~Francesco, editors, {\em Oxford
  Handbook of Random Matrix Theory}. Oxford University Press, 2011.
\newblock arXiv:0911.1153 [math.PR].

\bibitem[CIW19]{CaputoIoffe2019Confinement}
P.~Caputo, D.~Ioffe, and V.~Wachtel.
\newblock {Confinement of Brownian polymers under geometric area tilts}.
\newblock {\em Electronic Journal of Probability}, 24(none):1 -- 21, 2019.
\newblock arXiv:1809.03209 [math.PR].

\bibitem[CK01]{cerf2001low}
R.~Cerf and R.~Kenyon.
\newblock {The low-temperature expansion of the Wulff crystal in the 3D Ising
  model}.
\newblock {\em Comm. Math. Phys.}, 222:147--179, 2001.

\bibitem[CKP01]{CohnKenyonPropp2000}
H.~Cohn, R.~Kenyon, and J.~Propp.
\newblock A variational principle for domino tilings.
\newblock {\em Jour. AMS}, 14(2):297--346, 2001.
\newblock arXiv:math/0008220 [math.CO].

\bibitem[DFG19]{DiFran2019qvol}
P.~Di~Francesco and E.~Guitter.
\newblock A tangent method derivation of the arctic curve for q-weighted paths
  with arbitrary starting points.
\newblock {\em Jour. Phys. A}, 52(11):115205, 2019.
\newblock arXiv:1810.07936 [math-ph].

\bibitem[{\relax DLMF}]{NIST:DLMF}
{\it NIST Digital Library of Mathematical Functions}.
\newblock http://dlmf.nist.gov/, Release 1.1.8 of 2022-12-15.
\newblock F.~W.~J. Olver, A.~B. {Olde Daalhuis}, D.~W. Lozier, B.~I. Schneider,
  R.~F. Boisvert, C.~W. Clark, B.~R. Miller, B.~V. Saunders, H.~S. Cohl, and
  M.~A. McClain, eds.

\bibitem[DM15]{duse2015asymptotic}
E.~Duse and A.~Metcalfe.
\newblock {Asymptotic geometry of discrete interlaced patterns: Part I}.
\newblock {\em Intern. J. Math.}, 26(11):1550093, 2015.
\newblock arXiv:1412.6653 [math.PR].

\bibitem[DM20]{Duse2015_partII}
E~Duse and A.~Metcalfe.
\newblock {Asymptotic Geometry of Discrete Interlaced Patterns: Part II}.
\newblock {\em Annales de l'Institut Fourier}, 70(1):375--436, 2020.
\newblock arXiv:1507.00467 [math-ph].

\bibitem[Doo84]{doob1984classical}
J.L. Doob.
\newblock {\em Classical potential theory and its probabilistic counterpart}.
\newblock Springer, 1984.

\bibitem[Dys62]{dyson1962brownian}
F.J. Dyson.
\newblock {A Brownian motion model for the eigenvalues of a random matrix}.
\newblock {\em Jour. Math. Phys.}, 3(6):1191--1198, 1962.

\bibitem[FS23]{ferrari2023airy2}
P.L. Ferrari and S.~Shlosman.
\newblock {The Airy$_2$ process and the 3D Ising model}.
\newblock {\em Journal of Physics A: Mathematical and Theoretical},
  56(1):014003, 2023.
\newblock arXiv:2209.14047 [math.PR].

\bibitem[GH22]{gorin2022dynamical}
V.~Gorin and J.~Huang.
\newblock {Dynamical Loop Equation}.
\newblock {\em arXiv preprint}, 2022.
\newblock arXiv:2205.15785 [math.PR].

\bibitem[GO09]{Gnedin2009}
A.~Gnedin and G.~Olshanski.
\newblock A q-analogue of de {F}inetti's theorem.
\newblock {\em El. Jour. Combin.}, 16:R16, 2009.
\newblock arXiv:0905.0367 [math.PR].

\bibitem[Gor08]{Gorin2007Hexagon}
V.~Gorin.
\newblock Nonintersecting paths and the {H}ahn orthogonal polynomial ensemble.
\newblock {\em Funct. Anal. Appl.}, 42(3):180--197, 2008.
\newblock arXiv:0708.2349 [math.PR].

\bibitem[Gor21]{gorin2021lectures}
V.~Gorin.
\newblock Lectures on random lozenge tilings.
\newblock {\em Cambridge Studies in Advanced Mathematics. Cambridge University
  Press}, 2021.

\bibitem[GP19]{GorinPetrov2016universality}
V.~Gorin and L.~Petrov.
\newblock Universality of local statistics for noncolliding random walks.
\newblock {\em Ann. Probab.}, 47(5):2686--2753, 2019.
\newblock arXiv:1608.03243 [math.PR].

\bibitem[Hay56]{hayman1956genf}
W.K. Hayman.
\newblock {A generalization of Stirling’s formula}.
\newblock {\em J. Reine Angew. Math.}, 196:67–95, 1956.

\bibitem[Hua21]{huang2021beta}
J.~Huang.
\newblock {$\beta$-Nonintersecting Poisson Random Walks: Law of Large Numbers
  and Central Limit Theorems}.
\newblock {\em Intern. Math. Research Notices}, 2021(8):5898--5942, 2021.
\newblock arXiv:1708.07115 [math.PR].

\bibitem[Kas67]{Kasteleyn1967}
P.~Kasteleyn.
\newblock Graph theory and crystal physics.
\newblock In {\em {Graph Theory and Theoretical Physics}}, pages 43--110.
  Academic Press, London, 1967.

\bibitem[Ken97]{Kenyon1997LocalStat}
R.~Kenyon.
\newblock Local statistics of lattice dimers.
\newblock {\em Annales de Inst. H. Poincar\'e, Probabilit\'es et Statistiques},
  33:591--618, 1997.
\newblock arXiv:math/0105054 [math.CO].

\bibitem[Ken09]{Kenyon2007Lecture}
R.~Kenyon.
\newblock Lectures on dimers.
\newblock 2009.
\newblock arXiv:0910.3129 [math.PR].

\bibitem[KM59]{KMG59-Coincidence}
S.~Karlin and J.~McGregor.
\newblock Coincidence probabilities.
\newblock {\em Pacific J. Math.}, 9:1141--1164, 1959.

\bibitem[KO07]{OkounkovKenyon2007Limit}
R.~Kenyon and A.~Okounkov.
\newblock Limit shapes and the complex {B}urgers equation.
\newblock {\em Acta Math.}, 199(2):263--302, 2007.
\newblock arXiv:math-ph/0507007.

\bibitem[KOR02]{konig2002non}
W.~K{\"o}nig, N.~O'Connell, and S.~Roch.
\newblock {Non-colliding random walks, tandem queues, and discrete orthogonal
  polynomial ensembles}.
\newblock {\em Electron. J. Probab.}, 7(5):1--24, 2002.

\bibitem[KOS06]{KOS2006}
R.~Kenyon, A.~Okounkov, and S.~Sheffield.
\newblock Dimers and amoebae.
\newblock {\em Ann. Math.}, 163:1019--1056, 2006.
\newblock arXiv:math-ph/0311005.

\bibitem[LT15]{laslier2013lozenge}
B.~Laslier and F.~Toninelli.
\newblock {Lozenge tilings, Glauber dynamics and macroscopic shape}.
\newblock {\em Comm. Math. Phys}, 338(3):1287--1326, 2015.
\newblock arXiv:1310.5844 [math.PR].

\bibitem[Mut06]{mutafchiev2006size}
L.~Mutafchiev.
\newblock The size of the largest part of random plane partitions of large
  integers.
\newblock {\em Integers: Electronic Journal of Combinatorial Number Theory},
  6:A13, 2006.

\bibitem[Oko02]{Okounkov2002}
A.~Okounkov.
\newblock Symmetric functions and random partitions.
\newblock In S.~Fomin, editor, {\em Symmetric functions 2001: Surveys of
  Developments and Perspectives}. Kluwer Academic Publishers, 2002.
\newblock arXiv:math/0309074 [math.CO].

\bibitem[OR03]{okounkov2003correlation}
A.~Okounkov and N.~Reshetikhin.
\newblock {Correlation function of Schur process with application to local
  geometry of a random 3-dimensional Young diagram}.
\newblock {\em Jour. AMS}, 16(3):581--603, 2003.
\newblock arXiv:math/0107056 [math.CO].

\bibitem[OR07]{Okounkov2005}
A.~Okounkov and N.~Reshetikhin.
\newblock {Random skew plane partitions and the Pearcey process}.
\newblock {\em Commun. Math. Phys.}, 269(3):571--609, 2007.
\newblock arXiv:math/0503508 [math.CO].

\bibitem[Pet14]{Petrov2012}
L.~Petrov.
\newblock {Asymptotics of Random Lozenge Tilings via Gelfand-Tsetlin Schemes}.
\newblock {\em Probab. Theory Relat. Fields}, 160(3):429--487, 2014.
\newblock arXiv:1202.3901 [math.PR].

\bibitem[Pet15]{Petrov2012GFF}
L.~Petrov.
\newblock {Asymptotics of Uniformly Random Lozenge Tilings of Polygons.
  Gaussian Free Field}.
\newblock {\em {Ann. Probab.}}, 43(1):1--43, 2015.
\newblock arXiv:1206.5123 [math.PR].

\bibitem[Pet22]{petrov2022noncolliding}
L.~Petrov.
\newblock {Noncolliding Macdonald walks with an absorbing wall}.
\newblock {\em SIGMA}, 18:21, 2022.
\newblock arXiv:2204.09206 [math.PR].

\bibitem[She05]{Sheffield2008}
S.~Sheffield.
\newblock Random surfaces.
\newblock {\em Ast\'erisque}, 304, 2005.
\newblock arXiv:math/0304049 [math.PR].

\bibitem[TF61]{temperley1961dimer}
H.~Temperley and M.~Fisher.
\newblock Dimer problem in statistical mechanics - an exact result.
\newblock {\em Philos. Mag.}, 6(68):1061--1063, 1961.

\end{thebibliography}
\end{document}